\documentclass{icmart}
\usepackage[T1]{fontenc}
\usepackage{epsfig,color,floatflt}

\usepackage{amssymb,amsfonts,amsmath,stmaryrd}

\contact[bousquet@labri.fr]{CNRS, LaBRI, Universit\'e Bordeaux 1\\
351  cours de la Lib\'eration,
33405 Talence Cedex, France}

\newtheorem{theorem}{Theorem}[section]

\newtheorem{proposition}[theorem]{Proposition}

\theoremstyle{definition}
\newtheorem{definition}[theorem]{Definition}

\newtheorem{example}[theorem]{Example}

\def\emm#1,{{\em #1}}
\newcommand{\beq}{\begin{equation}}
\newcommand{\eeq}{\end{equation}}

\newcommand{\al}{\alpha}
\newcommand{\be}{\beta}
\newcommand{\la}{\lambda}
\newcommand{\si}{\sigma}
\newcommand{\eps}{\epsilon}

\newcommand{\cA}{{\cal A}}

\newcommand{\cC}{{\cal C}}
\newcommand{\cE}{{\cal E}}
\newcommand{\cL}{{\cal L}}
\newcommand{\cM}{{\cal M}}
\newcommand{\cW}{{\cal W}}
\newcommand{\cS}{{\cal S}}

\newcommand{\bE}{\overline{\cal E}}
\newcommand{\bD}{\overline{\cal D}}
\newcommand{\bF}{\overline{\cal F}}

\newcommand{\bv}{\bar v}

\newcommand{\bu}{\bar u}
\newcommand{\gf}{generating function}
\newcommand{\gfs}{generating functions}
\newcommand{\fsa}{finite state automaton}
\newcommand{\fps}{formal power series}
\newcommand{\pol}{polyomino}
\newcommand{\ragf}{rational [algebraic] GF}
\newcommand{\Blue}{}

\newcommand{\ns}{\mathbb{N}}
\newcommand{\ps}{\mathbb{P}}
\newcommand{\zs}{\mathbb{Z}}
\newcommand{\qs}{\mathbb{Q}}
\newcommand{\rs}{\mathbb{R}}
\newcommand{\cs}{\mathbb{C}}

\title[Rational and algebraic series in combinatorial
  enumeration]{Rational and  algebraic series\\
  in combinatorial enumeration} 

\author[Mireille Bousquet-M\'elou]{Mireille Bousquet-M\'elou
}

\begin{document}

\begin{abstract}
Let $\cA$ be a class of objects, equipped with an integer size such
that for all $n$ the number $a_n$ of objects of size $n$ is finite. We
are interested in the case where the \gf\ $\sum_n a_n t^n$ is rational, or
more generally
algebraic. This property has a practical interest, since one can
usually say a
lot on the numbers $a_n$, but also a combinatorial one: the
rational or algebraic nature of the \gf\ suggests that the objects
have a (possibly hidden) structure, similar to the \emm linear structure,
of words  in the rational case, and to
the  \emm branching structure, of trees in the algebraic case. We describe
and illustrate this combinatorial intuition, and discuss its validity.
While it seems to be satisfactory in the rational case, it is
probably incomplete in the algebraic one. We conclude with open questions.
\end{abstract}



\maketitle

\section{Introduction}
\label{section-intro}
The general topic of this paper is the enumeration of discrete
objects (words, trees, graphs...) and more specifically the
\emm rational,\/ or \emm algebraic, nature of the associated \gfs.
Let $\cA$ be a class of discrete objects equipped with a size:
$$
\begin{array}{rcl}
size :  \cA &\rightarrow &\ns \\
A& \mapsto& |A|.
\end{array}
$$
Assume that for all $n$, the number $a_n$ of objects of size $n$ is
finite. 
The \emm {\gf}  of the objects of, $\cA$, \emm counted by their
size,, is the following \fps\ in the indeterminate $t$:
\beq\label{GF-def}
A(t) := \sum_{n\ge 0} a_n t^n= \sum_{A\in \cA} t^{|A|}.
 \eeq
To take a very simple example, if $\cA$ is the set of words on the
alphabet $\{a,b\}$ and the size of a word is its number of letters,
then the \gf\ is $\sum_{n\ge 0} 2^n t^n=1/(1-2t)$. 

Generating functions provide both
a tool for solving counting problems, and a concise way to encode
their solution. Ideally, one would probably dream of finding
  a closed formula for the numbers $a_n$. But the world of
  mathematical objects  would be
  extremely poor if this was always possible. In practise, one is usually
 happy with an expression of the \gf \ $A(t)$, or even
  with a recurrence relation defining the sequence $a_n$, or  a
functional equation defining $A(t)$.   

Enumerative problems 
arise spontaneously in various fields of mathematics, computer
science, and physics. Among the most generous suppliers of such
problems, let us cite discrete probability theory, 
the analysis of the complexity of algorithms~\cite{knuth,fla-sedg1}, and the
discrete models  of statistical 
physics, like the famous Ising model~\cite{baxter-book}.
More generally, counting the  objects that occur in one's work
seems to answer a natural curiosity. It helps to understand the
objects, for instance to appreciate how restrictive are the conditions
that define 
them. It also forces us to get some understanding of the \emm
structure, of the objects: an enumerative result never comes for free,
but only after one has elucidated, at least partly, what the objects
really are.

We focus in this survey on objects having a rational, or, more
generally,  algebraic \gf. Rational and algebraic \fps\ are
well-behaved objects with many interesting properties. This is one of
the reasons why several classical textbooks on enumeration devote one
or several chapters to these
series~\cite{fla-sedg2,stanley-vol1,stanley-vol2}. These chapters give
typical examples of objects with a rational [resp. algebraic] \gf\
(GF). After a while, the collection of these examples builds up a
general picture: one starts thinking that yes, all these objects have
something in common in their structure. At the same time arises the
following question: do all objects with a \ragf\  look like that? In
other words, what does it mean, what does it suggest about the
objects when they are counted by a \ragf\ ?

This question is at the heart of this survey. For each of the two
classes of series under consideration, we first present a general
family of enumerative problems whose solution falls invariably in this
class. These problems are simple to describe: the first one deals
with walks in a directed graph, the other with plane
trees. Interestingly, these families of objects admit alternative
descriptions in language theoretic terms: they correspond to \emm regular
languages,, and to \emm unambiguous context-free languages,,
respectively.  The words of these languages have a clear recursive
structure, which explains directly the rationality [algebraicity] of
their GF. 

The series counting words of a regular [unambiguous context-free]
language are called $\ns$\emm-rational, [$\ns$\emm-algebraic,]. It is worth
noting that a rational [algebraic] series with non-negative
coefficients is not necessarily $\ns$-rational
[$\ns$-algebraic]. Since we want to appreciate whether our two generic
classes of objects are good representatives of objects with a \ragf,
the first question to address is the following: do we always fall in
the class of $\ns$-rational [$\ns$-algebraic] series when we count
objects with a \ragf? More informally, do these objects exhibit a
structure similar to the structure of regular [context-free]
languages?  Is such a  structure usually clearly visible? That is to
say, is it easy to feel, to predict rationality [algebraicity]?

We shall see that the answer to all these questions tends to be \emm
yes, in the rational case (with a few warnings...) but is probably
\emm no, in the algebraic case. In particular, the rich world of \emm
planar maps, (planar graphs embedded in the sphere) abounds in
candidates for non-$\ns$-algebraicity. The algebraicity of the
associated GFs has been known for more than 40 years (at least for
some families of maps), but it is only in the past 10 years that a
general combinatorial explanation of this algebraicity has
emerged. Moreover, the underlying constructions are more general that
those allowed in context-free descriptions, as they involve taking
\emm complements,.

Each of the main two sections ends with a list of questions. In
particular, we present at the end of Section~\ref{sec:alg} several
counting problems that are simple to state and have an algebraic GF,
but for reasons that remain mysterious. 

The paper is sometimes written in an informal style. We hope
that this will not stop the reader. We have tried to give precise
references where he/she will find more details and more
material on the topics we discuss. In particular, 
this survey borrows a lot to two books that we warmly recommend:
Stanley's \emm Enumerative
Combinatorics,~\cite{stanley-vol1,stanley-vol2}, and Flajolet \&
Sedgewick's \emm Analytic Combinatorics,~\cite{fla-sedg2}.

\medskip 
\noindent {\bf Notation and definitions.}
 Given a (commutative) ring $R$, we denote by $R[t]$ the
ring of polynomials in $t$ having coefficients in $R$. 
A \emm Laurent series, in $t$ 
is a series of the form $A(t)=\sum_{n\ge n_0} a_n
t^n$, with $n_0\in \zs$ and $a_n\in R$ for all $n$. If $n_0\ge0$, we
say that $A(t)$ is a \emm \fps,. The coefficient of $t^n$ is
denoted $a_n:=[t^n]A(t)$. The set of Laurent series forms a ring, and even
a field if $R$ is a field. The \emm quasi-inverse, of $A(t)$ is the
series $A^*(t):=1/(1-A(t))$. If $A(t)$ is a \fps\ with constant term 0, then
$A^*(t)$ is a \fps\ too.

In most occasions, the series we consider are GFs of the form~\eqref{GF-def}
 and thus have  rational
 coefficients. However, we sometimes consider refined enumeration
 problems, in which every object $A$ is \emm weighted,, usually by a
 monomial $w(A)$ in some additional indeterminates $x_1, \ldots , x_m$. The
 weighted GF is then
$\sum_{A\in \cA} w(A) t^{|A|}$, so that the coefficient ring is
 $\qs[x_1, \ldots , x_m]$ rather than $\qs$. 

We denote $\llbracket k\rrbracket=\{1,2, \ldots , k\}$. We use the standard notation $\ns, \zs, \qs$, and  $\ps:=\{1,2,3, \ldots \}$.

\section{Rational generating functions}

\subsection{Definitions and properties}
The Laurent series $A(t)$ with coefficients in the field $R$ is said
to be \emm rational,  if it can be written in the form   
$$
A(t)=\frac{P(t)}{Q(t)}
$$
where $P(t)$ and $Q(t)$ belong to $R[t]$.

There is probably no need to spend a lot of time explaining why such series
are simple and  well-behaved. We refer
to~\cite[Ch.~4]{stanley-vol1} and~\cite[Ch.~IV]{fla-sedg2} for a
survey of their  
properties. Let us  review briefly some of them, in the case where
$R=\qs$. The set of 
(Laurent) rational series is closed under sum, product, derivation,
reciprocals -- but \emm not under integration, as shown by
$A(t)=1/(1-t)$. The
coefficients $a_n$ of a rational series $A(t)$ satisfy a linear
recurrence relation with constant coefficients: for $n$ large enough,
$$
a_n=c_1a_{n-1}+c_2 a_{n-2} + \cdots + c_ka_{n-k}.
$$
The partial fraction expansion of $A(t)$ provides a closed form
expression of these coefficients of the form:
\beq
\label{rat-expr}
a_n= \sum_{i=0}^k P_i(n) \mu_i^n
\eeq
where the $\mu_i$ are the reciprocals of the roots of the denominator
$Q(t)$, and the $P_i$ are polynomials. In particular, if $A(t)$ has
non-negative integer coefficients, its radius of convergence $\rho$ is one its
the poles (Pringsheim) and the ``typical'' asymptotic behaviour
of $a_n$ is
\beq\label{asympt-rat}
a_n \sim \kappa \rho^{-n} n^d
\eeq
where
$d \in \ns$ and $\kappa$ is an algebraic number. The above
statement has to be taken with a grain of salt:   \emm all, poles of
minimal modulus may actually contribute to the dominant term in the
asymptotic expansion of $a_n$, as indicated by~\eqref{rat-expr}. 

Let us add that Pad\'e approximants allow us to \emm guess, whether a
\gf \ whose first coefficients are known is likely to be
rational. For instance, given the 10 first coefficients of the series
$$
A(t)=
 t+2\,{t}^{2}+6\,{t}^{3}+19\,{t}^{4}+61\,{t}^{5}+196\,{t}^{6}+629\,{t}
^{7}+2017\,{t}^{8}+6466\,{t}^{9}+20727\,{t}^{10}+O \left( {t}^{11}
 \right) ,
$$
it is easy to conjecture that actually
$$
A(t)={\frac {t \left( 1-t \right) ^{3}}{1-5\,t+7\,{t}^{2}-4\,{t}^{3}}}
.
$$
 Pad\'e approximants are implemented in most computer algebra
packages. For instance, the relevant {\sc Maple} command is {\tt
  convert/ratpoly}. 

\subsection{Walks on a digraph}
We now introduce our typical ``rational'' objects. 
Let $G=(V,E)$ be a directed graph with (finite) vertex set $V=
\llbracket p\rrbracket$
and (directed) edge set $E\subset V\times V$. A \emm  walk of length, $n$ on
$G$ is a sequence of vertices $w=(v_0, v_1, \ldots , v_n)$ such that for
all $i$, the  pair $(v_i, v_{i+1})$ is an edge. Such a walk
\emm goes from, $v_0$ \emm to, $v_n$. We denote $|w|=n$. 
Now assign to each directed edge
$e$ a weight (an indeterminate) $x_e$. Define the weight $x_w$ of the
walk $w$ as
the product of the weights of the edges it visits: more precisely,
$$
x_w = \prod_{i=0}^{n-1} x_{(v_i,v_{i+1})}.
$$
See Fig.~1(a) for an example. Let $X$ denote the (weighted)
adjacency matrix of $G$: for $i$ and $j$ in $\llbracket p\rrbracket$, the entry $X_{i,j}$
is $x_e$ if $(i,j)=e$ is 
an edge of $G$ and 0 otherwise. Let $W_{i,j}(t)$ be the weighted \gf\ of walks
going from $i$ to $j$:
$$
W_{i,j}(t)= \sum_{w : i \leadsto j} x_w t^{|w|}
.
$$
 It is well-known, and easy to prove, that $W_{i,j}$ is a rational
function in $t$ with coefficicients in $\qs[x_e, e\in
 E]$ (see~\cite[Thm. 4.7.1]{stanley-vol1}). 
\begin{theorem}
  The series $W_{i,j}(t)$ is the $(i,j)$-entry in the matrix $(1-tX)^{-1}$.
\end{theorem}

\begin{figure}[t]
\begin{center}
\input{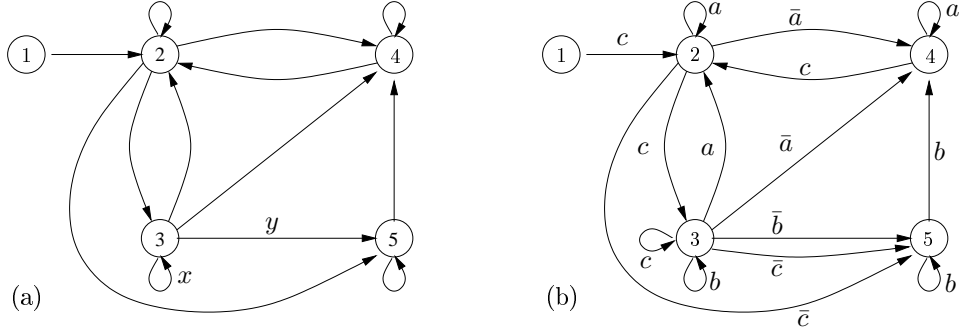}
\end{center}
\vskip -5mm\caption{(a) A weighted digraph. The default value of the weight is
  $1$. (b) A deterministic automaton on the alphabet $\{a,b,c,\bar
  a,\bar b,\bar 
  c\}$. The initial state is $1$ and and the final states  are $2$ and $3$.}
\label{fig:digraph}
\end{figure}

This theorem reduces the enumeration of walks on a digraph to the
calculation of the inverse of a matrix with polynomial
coefficients. It seems to be little known in the combinatorics
community that this inverse matrix can  be computed by studying the
elementary cycles of the digraph $G$.  This practical tool relies
 on Viennot's theory of \emm heaps of
pieces,~\cite{viennot-heaps}. Since it is little known, and often
convenient, let us  advertise it here. It will be illustrated further down.

An \emm elementary cycle, of $G$  is a
closed walk $w=(v_0, v_1, \ldots , v_{n-1},v_0)$ such that
$v_0,  \ldots , v_{n-1}$ are distinct. It is defined up to a cyclic
permutation of the 
$v_i$. That is, $(v_1, v_2, \ldots , v_{n-1}, v_0, v_1)$ is the same
  cycle as $w$.  A  collection $\gamma=\{\gamma_1, \ldots , \gamma_r\}$  of
(elementary) cycles is \emm non-intersecting, if the $\gamma_i$ are
  pairwise disjoint. The weight $x_\gamma$ of 
$\gamma$ is the product of the weights of the $\gamma_i$. We denote
  $|\gamma|=\sum|\gamma_i|$. 

\begin{proposition}[\cite{viennot-heaps}]\label{prop:viennot}
   The \gf\ of walks
going from $i$ to $j$ reads
$$
W_{i,j}(t)= \frac{N_{i,j}} D,
$$
where 
$$
D=\sum_{\gamma=\{\gamma_1, \ldots , \gamma_r\}} (-1)^r
x_{\gamma}t^{|\gamma|}
\quad\hbox{ and } \quad 
N_{i,j}=\sum_{w; \gamma=\{\gamma_1, \ldots , \gamma_r\}} (-1)^r  x_w
x_{\gamma}t^{|w|+ |\gamma|}.
$$
The polynomial $D$ is the alternating \gf\ of non-intersecting collections of 
cycles. In the expression of $N$, $\gamma$ a non-intersecting
collection of cycles and $w$ a 
self-avoiding walk going from $i$ to $j$, disjoint from the cycles
of $\gamma$. 
\end{proposition}
To illustrate this result, let us determine the \gf\ of walks going
from $1$ to  $2$ and from $1$ to $3$ on the digraph of
Fig.~\ref{fig:digraph}(a). This graph contains 4 cycles of length 1, 2
cycles of 
length 2, 2 cycles of length 3 and 1 cycle of length 4. By forming all
non-intersecting collections of cycles, one finds:
$$
  D(t)
=1-(3+x)t+(3+3x-2)t^2+(-1-3x+3+x-2)t^3+(x-1-x+1+x-y)t^4
$$
\vskip -7mm
$$\hskip -38mm
= 1-(3+x)t+(1+3x)t^2-2xt^3+(x-y)t^4.
$$
There is only one self-avoiding walk (SAW) going from 1 to 2, and one
SAW going from 1 to 3 (via the vertex 2).  The
collections of cycles that do not intersect these walks are formed
of loops, which gives
$$
N_{1,2}= t(1-t)^2(1-xt)\quad \hbox{and } \quad N_{1,3}= t^2(1-t)^2.
$$
Hence the \gf\ of walks that start from 1 and end at 2 or 3 is:
\beq\label{walk-enum}
W_{1,2} + W_{1,3}=\frac{N_{1,2} + N_{1,3}}{D}= \frac{t(1-t)^2(1+t-xt)}{
  1-(3+x)t+(1+3x)t^2-2xt^3+(x-y)t^4}.
\eeq

\subsection{Regular languages and automata}

There is a very close connection between the collection of walks on a
digraph and the words of
\emm regular languages,.  
%
Let $\cal A$ be an \emm alphabet,, that is, a finite set of symbols
(called \emm letters)., A \emm word, on $\cal A$ is a sequence $u=u_1 u_2
\cdots u_n$ of letters. The number of occurrences of the letter $a$ in
the word $u$ is denoted $|u|_a$. The \emm product, of two words $u_1 u_2
\cdots u_n$ and $v_1 v_2 \cdots v_m$ is  the concatenation 
 $u_1 u_2 \cdots u_nv_1 v_2 \cdots v_m$. The empty word is denoted
$\eps$. A \emm language, on $\cal A$ is a set of words. We define two
operations on  languages:

\begin{itemize}
 
\item[--] the product $\cal L K$ of two languages $\cal L$ and $\cal K$ is
   the set of words  $uv$, with $u\in\cal L$ and
   $v\in\cal K$; this product is easily seen to be associative,

\item[--]  the star $\cal L^*$ of the language $\cal L$ is the union of all
languages ${\cal L}^k$, for $k \ge 0$. By convention,  ${\cal L}^0$ is
reduced to the empty word $\eps$.
\end{itemize}

\noindent A \emm\fsa,  on $\cal A$ 
is a digraph $(V,E)$ \emm with possibly multiple edges,, together with:

-- a labelling of the edges  by letters of $\cal A$, that is to say, a
   function $L: E \rightarrow {\cal A}$,

-- an initial vertex $i$,

-- a set $V_f \subset V$ of final vertices.

\noindent 
 The vertices are usually called the \emm states, of
the automaton. The automaton is  \emm deterministic, if for every
state $v$ and every letter $a$, there is at most one edge labelled $a$
starting from $v$. 


To every walk on the underlying multigraph, one associates a word on
the alphabet $\cal A$ by reading the letters met along
the walk. The  language ${\cal L}$ \emm recognized, by the
automaton is 
the set of words associated with walks going from the
initial state $i$ to one of the states of $V_f$. 
 For $j\in V$, let ${\cal
  L}_j$ denote the set of words associated with walks going from $i$
to $j$. These sets admit 
a recursive description.
For  the automaton of
Fig.~\ref{fig:digraph}(b), one has $ {\cal L} = {\cal L}_2\cup {\cal
  L}_3$ with
$$
\begin{array}{llllllllllllllllllll}
{\cal L}_1& =&\{\eps\},&&&&&&&\\
{\cal L}_2& =&{\cal L}_1c \cup {\cal L}_2a \cup {\cal L}_3a\cup
{\cal L}_4c,&&&
{\cal L}_4& =&{\cal L}_2\bar a \cup {\cal L}_3\bar a \cup {\cal L}_4a \cup {\cal L}_5b,\\
{\cal L}_3& =&{\cal L}_2c\cup {\cal L}_3b\cup {\cal L}_3c,&&&
{\cal L}_5& =&{\cal L}_2 \bar c\cup{\cal L}_3\bar b\cup  {\cal L}_3\bar c
\cup {\cal L}_5b. 
\end{array}
$$

Remarkably, there also exists a non-recursive
combinatorial description  of the languages that are recognized by an
automaton~\cite [Thms. 3.3 and 3.10]{hopcroft}.

\begin{theorem}\label{th:regular}
    Let  $\cal L$ be a language on the alphabet $\cal A$. There exists a \fsa\
   that recognizes $\cal L$ if and only if
 $\cal L$ can be expressed in terms of finite languages on $\cA$,
using a finite number of unions,    products and stars of languages.
 
If these conditions hold, $\cal L$ is said to be \emm
regular,. Moreover, there exists a \emm deterministic, automaton that
recognizes $\cal L$.
\end{theorem}

\noindent{\bf Regular languages and walks on digraphs.}
  Take a deterministic automaton, and associate with it a weighted
digraph as follows: the vertices are those of the automaton, and  for
all vertices $j$ and $k$, if $m$  edges  go from $j$ to
$k$ in the automaton, they are replaced by a \emm single, edge labelled $m$
in the digraph.
For instance, the  automaton of
Fig.~\ref{fig:digraph}(b) gives the digraph to its left, with
$x=y=2$.  Clearly, the length GF of words of $\cal L$ is the
GF of (weighted) walks of this digraph going from the initial vertex
$i$ to one of the final vertices of $V_f$. For instance, according
to~\eqref{walk-enum}, the length
GF of the language recognized by the automaton of
Fig.~\ref{fig:digraph}(b) is
\beq\label{eq:regular}
A(t)={\frac {t \left( 1-t \right) ^{3}}{1-5\,t+7\,{t}^{2}-4\,{t}^{3}}}
.
\eeq

  Take a regular language $\cL$ recognized by a  deterministic
  automaton $\cA$. There  exists another   deterministic automaton
  that recognizes $\cL$ and does not contain multiple
  edges. The key  is to create a state $(j,a)$ for every edge labelled
  $a$   ending at $j$ in   the automaton $\cA$. The digraph associated
  with this new automaton has all its edges labelled 1, so that there
  exists a length preserving bijection between the words of $\cL$ and
  the walks on the digraph going from a 
   specified initial vertex $v_0$ to one of the vertices of a given
   subset $V_f$ of vertices.

Conversely, starting from a digraph with all
edges labelled 1, together with a specified  vertex $v_0$ and a set
$V_f$ of final vertices, it is easy to construct a regular language
that is in bijection with the walks of the graph going from $v_0$ to
$V_f$ (consider the automaton obtained by labelling all edges with
distinct letters). This shows that
\emm  counting words of regular languages is completely equivalent to
  counting walks in  digraphs.,
In particular, the set of rational series obtained in both types of
problems coincide, and have even been given a name:
\begin{definition}
  A series $A(t)=\sum_{n\ge 0} a_n t^n$ with  coefficients in $\ns$
is said to be $\ns$-\emm rational, if there exists a
  regular language having  \gf\ $A(t)-a_0$.
\end{definition}
The description of regular languages given by Theorem~\ref{th:regular}
implies that the set of 
$\ns$-rational series contains the smallest set of series containing
$\ns[t]$ and closed under sum, product and quasi-inverse. The converse
is true~\cite[Thm.~II.5.1]{salomaa}.
There exists a simple way to decide whether a given rational series
with coefficients in $\ns$ is
$\ns$-rational~\cite[Thms. II.10.2 and II.10.5]{salomaa}.
\begin{theorem}
\label{thm:soittola}
   A series $A(t)=\sum_{n\ge 0} a_n t^n$ with  coefficients in $\ns$ 
is  $\ns$-rational if and only if there exists a positive
integer $p$ such that for all $r\in\{0,\ldots, p\}$,
the series
$$
A_{r,p}(t):=\sum_{n\ge 0} a_{np+r}t^n$$
has a unique singularity of 
minimal modulus (called \emm dominant,).
\end{theorem}

There exist rational series with non-negative integer
coefficients that are \emm not, $\ns$-rational. For instance, let
$\alpha$ be such that $\cos \alpha=3/5$ and $\sin \alpha= 4/5$, and
define $a_n=25^n \cos(n\alpha)^2$. It is not hard to see that $a_n$ is
a non-negative integer. The associated series $A(t)$ reads
$$
A(t)=\frac{1-2t+225t^2}{(1-25t)(625t^2+14t+1)}.
$$
It has $3$ distinct dominant
poles. As $\alpha$ is not a rational multiple of
$\pi$, the same holds for all series $A_{0,p}(t)$, for all values of
$p$. Thus $A(t)$ is not $\ns$-rational. 

\subsection{The combinatorial intuition of rational \gfs}
\label{sec:intuition-rational}
We have described two families of combinatorial objects that
naturally yield rational \gfs : walks in a digraph and words of
regular languages. We have, moreover, shown that the enumeration of
these objects are  equivalent problems. It seems that these families
convey the ``right'' intuition about objects with a rational GF. By
this, we mean informally that:
\begin{itemize}
  \item[$(i)$] ``every'' family of objects with a rational
  GF has actually  an $\ns$-rational GF,
\item[$(ii)$] for almost all families of combinatorial
  objects with a rational GF, it is easy to foresee that there will
  be a bijection between these objects and words of a regular language.
\end{itemize}
Point $(ii)$ means that most of these families $\cal F$ have a clear
\emm automatic structure,, similar to the automatic  structure of
regular languages: 
roughly speaking, the objects of $\cal F$ can be constructed
recursively using unions of sets and concatenation of \emm cells,
(replacing letters). A more formal definition would simply paraphrase
the definition of automata.

Point $(i)$ means simply that I have never met a
counting problem that would yield a rational, but not $\ns$-rational
GF. This includes problems coming from algebra, like growth functions of
groups. On the contrary, Point $(ii)$ only concerns purely combinatorial
problems (but I do not want to be asked about the border between
combinatorics and algebra). It admits very few
counter-examples. Some will be discussed in
Section~\ref{sec:rat-difficult}. 
For the moment, let us illustrate the two above statements by
describing the automatic structure of certain classes of objects (some
being rather general), borrowed from~\cite[Ch. 4]{stanley-vol1}.

\subsubsection{Column-convex polyominoes}
A \emm polyomino, is a finite union of cells of the square lattice,
whose interior is connected. Polyominoes are considered up to a
translation. A polyomino is \emm column-convex, (cc) if its intersection
with every vertical line is connected. Let $a_n$ be the number of
cc-polyominoes having $n$ cells, and let $A(t)$ be the associated
\gf. We claim that  these polyominoes have an automatic structure.

\begin{figure}[t]
\begin{center}
\input{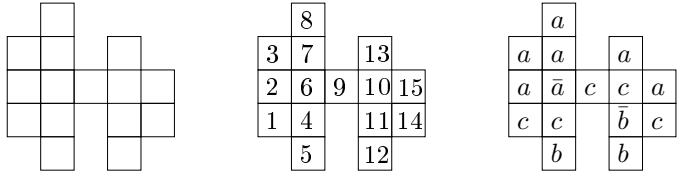}
\end{center}
\vskip -5mm\caption{A column-convex polyomino, with the numbering and encoding of
  the cells.} 
\label{fig:column}
\end{figure}

Consider a cc-\pol\ $P$  having $n$ cells. Let us number these cells from 1
to $n$ as illustrated in Fig.~\ref{fig:column}. The 
columns are visited from left to right. In the first column, cells are
numbered from bottom to top. In each of the other columns, the lowest
cell that has a left neighbour gets the smallest number; then the cells
lying below it are numbered from top to bottom, and finally the cells
lying above it are numbered from bottom to top.  Note that for all $i$,
the cells labelled $1,2, \ldots , i$ form a cc-\pol. This will be
essential in our description of the automatic structure of these
objects.
Associate with $P$ the
word $u=u_1 \cdots u_n$ on the alphabet $\{a,b,c\}$  defined by

-- $u_i=c$ (like Column) if the $i$th cell is the first to be visited in
its column,

-- $u_i=b$ (like Below) if the $i$th cell lies below the first visited
cell of its column,

-- $u_i=a$ (like Above) if the $i$th cell lies above the first visited
cell of its column.

\noindent Then, add a bar on the letter $u_i$ if the $i$th cell of $P$
has a South
neighbour, an East neighbour, but no South-East neighbour. (In other words,
the barred letters indicate where to start a new column, when the
bottommost cell of this new column lies above the bottommost cell of
the previous column.) This gives a word $v$ on the alphabet
$\{a,b,c,\bar a , \bar b ,\bar c\}$. It is not hard to see that the
map that sends $P$ on the word $v$ is a size-preserving bijection
between cc-\pol es and words recognized by the automaton of
Fig.~\ref{fig:digraph}(b). Hence by~\eqref{eq:regular}, the \gf\ of
column-convex polyominoes is~\cite{temperley}:
$$
A(t)={\frac {t \left( 1-t \right) ^{3}}{1-5\,t+7\,{t}^{2}-4\,{t}^{3}}}
.
$$

\subsubsection{P-partitions}
A \emm partition, of the integer $n$ into at most $k$ parts is
a non-decreasing  $k$-tuple  $\la=(\la_1, \ldots ,
\la_k)$ of nonnegative integers that sum to $n$. 
This classical number-theoretic notion is
generalized by the notion of \emm P-partitions., Let $P$ be a
\emm natural, partial order on $\llbracket k\rrbracket$ (by \emm
natural, we mean  that if $i<j$ in $P$, then $i<j$
in $\ns$). A \emm $P$-partition, of $n$ is a $k$-tuple $\la=(\la_1, \ldots , 
\la_k)$ of nonnegative
integers that sum to $n$ and satisfy $\la_i\le \la_j$ if $i\le j$ in
$P$. Thus when $P$ is the natural total order on $\llbracket
k\rrbracket$, a $P$-partition 
is simply a partition\footnote{A $P$-partition is usually defined as
  an \emm order-reversing, map from $\llbracket k\rrbracket$ to
  $\ns$~\cite[Section 
    4.5]{stanley-vol1}.  Both notions are of course completely
  equivalent.}.

We are interested in the following series:
$$
F_P(t)=\sum_{\la} t^{|\la|},
$$
where the sum runs over all $P$-partitions 
and
$|\la|=\la_1+\cdots +\la_k$ is the \emm weight, of $\la$.  

The case of ordinary partitions is easy to analyze: every partition
can be written in a unique way as a linear combination
\beq\label{linear-partitions}
c_1\la^{(1)} + \cdots +c_k\la^{(k)} 
\eeq
where $\la^{(i)}= (0,0,\ldots , 0, 1, 1, \ldots , 1)$ has exactly $i$
parts equal to $1$ and $c_i \in \ns$.  The weight of $\la^{(i)}$
is $i$, and one obtains:
\beq\label{part-gf}
F_P(t)=\frac 1 {(1-t)(1-t^2)\cdots(1-t^k)}.
\eeq
The automatic structure of (ordinary) partitions is transparent: since
they are constructed by adding a number of copies of $\la^{(1)}$, then
a number of copies of $\la^{(2)}$, and so on, there is a size
preserving bijection between these partitions and walks starting from
$1$ and ending anywhere
in the following digraph:
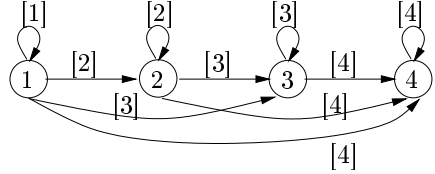
\begin{center}
\input{partition-automaton.pstex_t}  
\end{center}
Note that this graph corresponds to $k=4$, and that an edge labelled
$[\ell]$ must be understood as a \emm sequence, of $\ell$ edges.
These labels do \emm not, correspond to multiplicities.
Observe that the only cycles in this digraph are loops. This, combined with
Proposition~\ref{prop:viennot}, explains the factored form of the
denominator of~\eqref{part-gf}.

 Consider now the partial order on $\llbracket 4\rrbracket$ defined by $1<3,
 2<3$ and $2<4$. The
partitions of weight at most 2 are
$$
(0,0,0,0), (0,0,1,0), (0,0,0,1), (1,0,1,0), (0,0,1,1), (0,0,2,0),
(0,0,0,2),
$$
so that $F_P(t)=1+2t+4t^2+O(t^3)$. If one is brave enough to list
$P$-partitions of weight at most 20, the Pad\'e approximant of the
truncated series thus obtained is remarkably simple:
$$
F_P(t)=\frac{1+t+t^2+t^3+t^4}{(1-t)(1-t^2)(1-t^3)(1-t^4)}+O(t^{21}),
$$
and allows one to make a (correct) conjecture.


It turns out that the \gf\ of $P$-partitions is always a rational
series of denominator $(1-t)(1-t^2) \ldots (1-t^k)$. Moreover,
$P$-partitions obey our general intuition  about objects with a
rational GF. 
The following proposition, illustrated below by an example, describes
their automatic 
structure: the set of $P$-partitions
can be partitioned  into a finite number of subsets; in each of
these subsets, partitions have a structure similar
to~\eqref{linear-partitions}. Recall that a \emm linear extension, 
of $P$ is a 
bijection $\sigma$ on $\llbracket k\rrbracket$ such that
$\sigma(i)<\sigma(j)$ if $i<j$ in $P$.
\begin{proposition}[\cite{stanley-vol1}, Section 4.5]
Let $P$ be a natural order on $\llbracket k\rrbracket$.

For every $P$-partition $\la$, there exists a unique linear extension
  $\sigma$ of $P$ such that for all $i$, $\la_{\sigma(i)} \le
  \la_{\sigma(i+1)}$, the inequality being  strict if
  $\sigma(i)>\sigma(i+1)$.
We say that $\la$ is \emm compatible, with $\sigma$.

Given a linear extension $\sigma$, the  $P$-partitions that are
compatible with $\sigma$ can be written in a unique way as a linear
combination with coefficients in $\ns$:
\beq\label{P-part-structure}
\la^{(\si,0)}+ c_1\la^{(\si,1)} + \cdots +c_k\la^{(\si,k)} 
\eeq
where $\la^{(\si,0)}$ is the smallest $P$-partition compatible with
$\si$:
$$
\la^{(\si,0)}_{\si(j)}= \left| \{i<j: \si(i)>\si(i+1)\}\right| \quad \hbox{
  for } 1\le j \le k,
$$
and for $1\le i\le k$,
$$
  (\la^{(\si,i)}_{\si(1)}, \ldots,  \la^{(\si,i)}_{\si(k)})=  (0,0,\ldots , 0, 1, 1,
  \ldots , 1)$$ has exactly $i$ 
parts equal to $1$.
Thus the GF of these  $P$-partitions is
$$
F_{P,\sigma}(t)=\frac{t^{e(\sigma)}}{(1-t)(1-t^2) \ldots (1-t^k)}
$$
where $e(\sigma)$ is a variant of the \emm Major index, of  $\sigma$:
$$
e(\sigma)=\sum_{i : \sigma(i)>\sigma(i+1)} (k-i).
$$
\end{proposition}
\noindent{\bf Example.} Let us return to the order $1<3$, $2<3$ and
$2<4$. The 5 linear 
extensions  are $1234, 2134, 1243, 2143$ and $2413$. Take
$\sigma=2143$. The $P$-partitions $\la$ that are compatible with $\sigma$
are those that satisfy $\la_2<\la_1\le \la_4<\la_3$. The smallest of
them is thus $\la^{(\si,0)}=(1,0,2,1)$. Then
$\la^{(\si,1)}=(0,0,1,0)$, $\la^{(\si,2)}=(0,0,1,1)$,
$\la^{(\si,3)}=(1,0,1,1)$ and $\la^{(\si,4)}=(1,1,1,1)$.

\subsubsection{Integer points in a convex polyhedral
  cone~(\cite{stanley-vol1}, Sec.~4.6)}
\label{sec:cone}
Let $\cal H$ be a finite collection of 
linear half-spaces of $\rs^m$ of the form $c_1\al_1+\cdots +
c_m\al_m\ge 0$, with $c_i\in \zs$.
We are interested in the set $\cE$ of 
\emm non-negative integer points, 
 $\al=(\al_1, \ldots , \al_m)$
lying in the intersection of those half-spaces.
For instance, we could
have the following set $\cal E$, illustrated in Fig.~\ref{fig:polyhedron}(a):
\beq
\label{cone-example}
{\cal E} =\{(\al_1, \al_2)\in \ns^2: 2\al_1 \ge \al_2 \hbox{ and }
2\al_2 \ge \al_1\}.
\eeq


\begin{figure}[hbt]
\begin{center}
\input{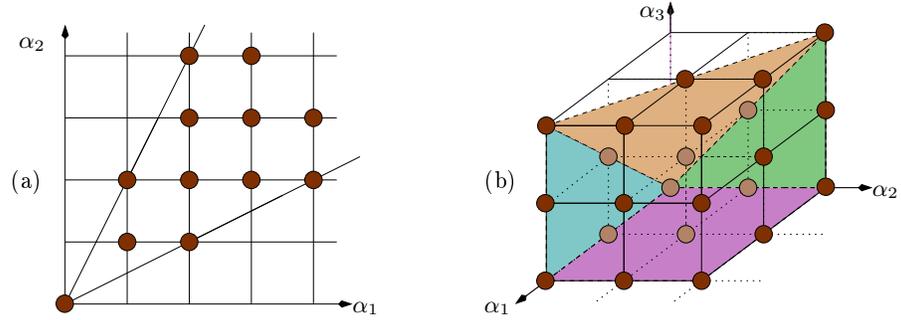}
\end{center}
\vskip -5mm\caption{Integer points in a  polyhedral cone.
} 
\label{fig:polyhedron}
\end{figure}

Numerous enumerative problems (including $P$-partitions) can be formulated in terms of
linear inequalities as above. The \gf\ of $\cal E$ is
$$
E(t)= \sum_{\al \in \cal E} t^{|\al|},
$$
where $|\al|= \al_1+\cdots +\al_m$. In the above example,
$
E(t)=1+t^2+2t^3+t^4+2t^5+3t^6+2t^7+O(t^8).
$

The set $\cal E$ is a monoid (it is closed under summation).  In
general, it is not a \emm free, 
monoid. Geometrically, the set $\cal C$ of non-negative \emm
real, points in the intersection of the half-spaces of $\cal H$ forms
a \emm pointed convex polyhedral cone, (the term \emm pointed, means that
it does not contain a line), and  $\cal
E$ is the set of integer points in $\cC$.

\medskip
\noindent{\bf The simplicial case.} 
In the simplest case, the 
cone $\cal C$  is simplicial. This implies
that the \emm monoid, $\cE$ is
\emm simplicial,, meaning that there exists linearly independent vectors $\al^{(1)},
\ldots , \al^{(k)}$ such that
$$
{\cal E}= \left\{ \al \in \ns^m: \al= q_1 \al^{(1)} + \cdots + q_k
\al^{(k)} \hbox{     with } q_i \in \qs, \ q_i \ge 0\right\}.
$$
This is the case in  Example~\eqref{cone-example}, with
$\al^{(1)}=(1,2)$ and $\al^{(2)}=(2,1)$.
The \emm interior, of $\cal E$ (the set of points of $\cE$ that are
not on the boundary of $\cC$) is then
\beq\label{interior}
\bE= \left\{ \al \in \ns^m: \al= q_1 \al^{(1)} + \cdots + q_k \al^{(k)} \hbox{
    with } q_i \in \qs, q_i > 0\right\}.
\eeq
Then 
there
exists  a finite subset $\cal D$ of 
$\cal E$ [resp.  $\bD$ of $\bE$] such that every element of $\cal E$
[resp.~$\bE$] can be written uniquely in the form
\beq\label{poly-structure} 
\al= \be + c_1\al^{(1)} +\cdots +  c_k\al^{(k)},
\eeq
with $\beta \in \cal D$ [resp. $\beta \in \bD$] and $c_i \in
\ns$~\cite[Lemma 4.6.7]{stanley-vol1}. In our running example~\eqref{cone-example},  taken with
$\al^{(1)}=(1,2)$ and $\al^{(2)}=(2,1)$, one has ${\cal D}=
  \{(0,0),(1,1),(2,2)\}$ while $\bD=  \{(1,1),(2,2),(3,3)\}$.
Compare~\eqref{poly-structure}  with the
structure found for 
$P$-partitions~\eqref{P-part-structure}. Thus  $\cal E$ and $\bE$ have
an automatic structure and their GFs read
$$
E(t)= \frac{\sum_{\be \in \cal D} t^{|\be|} }{\prod_{i=1}^k
  (1-t^{|\al^{(i)}|})}
\quad \left[\hbox{resp. } \quad \bar E(t)= \frac{\sum_{\be \in  \bD} t^{|\be|} }{\prod_{i=1}^k
  (1-t^{|\al^{(i)}|})} \right].
$$
In Example~\eqref{cone-example}, one thus obtains
$$
E(t)=\frac{1+t^2+t^4}{(1-t^3)^2}=\frac{1-t+t^2}{(1-t)(1-t^3)} \quad
\hbox{ and }\quad \bar E(t)=t^2E(t).
$$

\medskip
\noindent{\bf The general case.} 
 The set 
$\cal E$ can always be partitioned into a finite number of sets $\bF$
 of the form~\eqref{interior}, 
where $\cal F$ is a simplicial monoid~\cite[Ch. 4,   Eq.~(24)]{stanley-vol1}. 
Thus $\cE$,
as a finite union of sets with an automatic
 structure, has an automatic structure as well. The associated
\gf\ $E(t)$ is $\ns$-rational, with a denominator which is a product of
cyclotomic polynomials. 

Consider, for example, the set
$$
{\cal E}=\{(\al_1,\al_2,\al_3) \in \ns^3 : \al_3\le \al_1+\al_2\}.
$$
The cone $\cal C$ of non-negative \emm real, points $\al$ satisfying $\al_3\le
\al_1+\al_2$ is \emm not,  simplicial, as it has 4 faces of
dimension 2, lying respectively in the hyperplanes $\al_i=0$ for
$i=1,2,3$ and $\al_3= \al_1+\al_2$ (Fig.~\ref{fig:polyhedron}(b)). 
But it is the
union of two simplicial cones ${\cal C}_1$ and  ${\cal C}_2$, obtained
by intersecting $\cal C$ with 
the half-spaces $\al_1\ge \al_3$ and  $\al_1\le \al_3$,
respectively. Let $\cE_1$ [resp. $\cE_2$] denote the
set of integer points of ${\cal C}_1$ [resp.  ${\cal C}_2$].

The fastest way to obtain the \gf\ $E(t)$ is to write
\beq\label{count-fast}
E(t)=E_1(t)+E_2(t)-E_{12}(t)
\eeq
where $E_{12}(t)$ counts integer points in the intersection of  ${\cal
  C}_1$ and  ${\cal C}_2$ (that is, in the plane $\al_1= \al_3$). Since ${\cal
  E}_1$,  ${\cal E}_2$ and  ${\cal
  E}_1\cap {\cal E}_2$ are  simplicial cones (of dimension 3, 3 and 2
respectively), the method presented above for simplicial cones
applies. Indeed,
  ${\cal   E}_1$ [resp. ${\cal   E}_2$;  ${\cal   E}_{12}$] is the set
of linear combinations (with  coefficients in $\ns$) of $(1,0,1)$,
$(0,1,0)$ and $(1,0,0)$ [resp. 
  $(1,0,1)$, $(0,1,0)$ and $(0,1,1)$; 
  $(1,0,1)$ and $(0,1,0)$].
This implies:
$$
E(t)= \frac 1{(1-t)^2(1-t^2)} + \frac 1{(1-t)(1-t^2)^2} - \frac
1{(1-t)(1-t^2)} =\frac{1+t+t^2}{(1-t)(1-t^2)^2}.
$$
However, the ``minus'' sign in~\eqref{count-fast} prevents us from seeing
directly the automatic nature of $\cal E$ (the difference of
$\ns$-rational series is not always $\ns$-rational). This
structure only becomes clear when we write $\cal E$ as the disjoint
union of the interiors of all simplicial monoids induced by the triangulation
of $\cal C$ into  ${\cal  C}_1$ and   ${\cal C}_2$. These monoids are
the integer points of the faces (of all possible dimensions) of
${\cal  C}_1$ and 
${\cal C}_2$. As there are 12 such faces (more precisely, 1 [resp. 4,
  5, 2] faces of dimension 0 [resp. 1, 2, 3]), this gives $\cal E$ as
the disjoint union of 12 sets having an automatic structure of the
form~\eqref{interior}. 

\subsection{Rational \gfs: more difficult questions}
\label{sec:rat-difficult}
\subsubsection{Predicting rationality}\label{sec:predicting}
We wrote in Section~\ref{sec:intuition-rational} that it is usually easy to foresee, to
predict when a class of combinatorial objects has  a rational
GF. There are a few  exceptions. Here is
one of the most remarkable ones.

\begin{example}[{\bf Directed animals}]
A \emm directed animal with a compact source of size, $k$ is a finite set of
points $A$ on the square lattice $\zs^2$ such that:

-- the points $(-i,i)$ for $0\le i <k$ belong to $A$; they are called
   the \emm source points,,

-- all the other points in $A$ can be reached from one of the source
   points by a path made of North and East steps, having all its
   vertices in $A$. 

\noindent See Fig.~\ref{fig:compact} for an illustration. A similar
notion 
exists for the triangular
lattice. It turns out that these animals have extremely simple \gfs~\cite{gouyou-viennot,betrema-penaud}.

\begin{figure}[t]
\begin{center}
\input{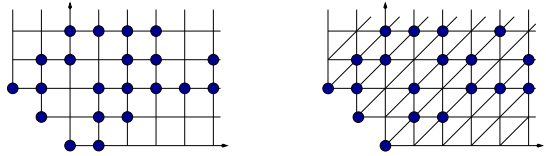}
\end{center}
\vskip -5mm\caption{Compact-source directed animals on the square and triangular
  lattices.}  
\label{fig:compact}
\end{figure}

\begin{theorem}\label{thm:compact}
  The number of compact-source directed animals of cardinality $n$ is
  $3^{n-1}$ on the square lattice, and $4^{n-1}$ on the triangular
  lattice. 
\end{theorem}
The corresponding GFs are respectively $t/(1-3t)$
and $t/(1-4t)$, and are as rational as a series can be. There is at
the moment no simple combinatorial intuition as to why these animals
have rational GFs.  A
bijection between square lattice
animals and words on a 3-letter alphabet was described
in~\cite{gouyou-viennot}, but it does not shed a clear 
light on the structure 
of these objects. Still, there is now a
convincing explanation of the \emm algebraicity, of these
series (see Section~\ref{sec:diriges}).

\end{example}
\smallskip

\begin{example}[{\bf The area under Dyck paths}]\label{ex:area}
  Another family of (slightly less natural) examples
is provided by the enumeration of points lying below certain lattice
paths. For instance, let us call \emm Dyck path of length, $2n$ 
  any path $P$ on $\zs^2$ formed of steps 
$(1,1)$ and $(1,-1)$, that starts  from $(0,0)$ and ends at $(2n,0)$
without ever hitting a point with a negative ordinate. The \emm area,
below $P$ is the number of non-negative integer points $(i,j)$, with
$i\le 2n$, lying weakly below $P$ (Fig.~\ref{fig:dyck-area}). It
turns out that the sum of the areas of 
Dyck paths of length $2n$ is simply
$$
\sum_{P: |P|=2n} a(P)=4^n.
$$
Again, the rationality of the associated \gf\ does not seem easy to
predict, but there are good combinatorial reasons explaining why it is
\emm
algebraic,. See~\cite{chottin-cori,pergola-moments}
 for a direct explanation of this 
result, references, and a few variations on this phenomenon, first
spotted by Kreweras~\cite{kreweras-aire}.
\end{example}

\begin{figure}[hbt]
\begin{center}
\input{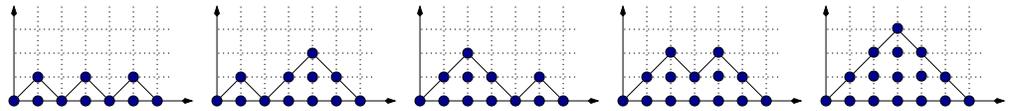}
\end{center}
\vskip -5mm\caption{The 5 Dyck paths of length 6 and the $4^3=64$ points lying below.}  
\label{fig:dyck-area}
\end{figure}

Finally, let us mention that our optimistic statement about how easy
it is to predict the rationality of a \gf\ becomes less and less true
as we move from purely combinatorial problems to more algebraic
ones. For instance, it is not especially easy to foresee that a
 group
has an automatic structure~\cite{epstein}. 
%
Let us give also an example coming from number theory. Let
$P(x)\equiv P(x_1, \ldots , x_r)$ be a polynomial with integer coefficients, and
take $p$ a prime. For $n\ge 0$, let $a_n$ be the number of $x \in 
(\zs/p^n\zs)^r$ such that $P(x)=0 \mod p^n$.  Then the \gf\ $\sum_n a_n
t^n$ is rational. A related result holds with $p$-adic
solutions~\cite{denef,igusa}.  

\subsubsection{Computing a rational \gf}
Let us start with an elementary, but important observation. Many
enumerative problems, including some very hard, can be \emm approximated, by
problems having a rational GF. To take one example, consider the
notoriously difficult problem of counting \emm self-avoiding polygons,
(elementary cycles)   
 on the square lattice. It is easy to convince
oneself that the \gf\ of SAP lying in a horizontal strip of height
$k$ is rational for all $k$. This does not mean that it will be easy
(or even possible, in the current state of affairs) to compute the
corresponding \gf\  when $k=100$. Needless to say, there is at the
moment no hope to express this GF for a \emm generic, value of
$k$. The \gf\ of SAP having $2k$
horizontal steps can also be seen to be rational. Moreover, these SAP
can be described in terms of linear inequalities (as in
Section~\ref{sec:cone}), which implies that the denominator of the
corresponding series $G_k$ is a product of cyclotomic polynomials. But
again, no one knows what this series is for a generic value of $k$, or
even for $k=100$. Still, some progress have been made 
recently, since it has been proved that the series $G_k$ have more and
more poles as $k$ increases, which means that their denominators
involve infinitely many cyclotomic polynomials~\cite{rechnitzer-sap}. This
may be considered as a \emm proof of the difficulty, of this enumerative
problem~\cite{guttmann-solvability}.  

\medskip
In general, computing the (rational) \gf\ of a family of objects
depending on a parameter $k$ may be non-obvious, if not difficult,
even if the objects are 
clearly regular, and even if the final result turns out to be nice. A
classical example is provided by the growth functions of 
 Coxeter groups~\cite{macdonald}. Here is a more
 combinatorial example. A partition $\la=(\la_1,
\ldots , \la_k)$ 
is said to be a $k$\emm-Lecture Hall partition, ($k$-LHP) if
$$
0\le \frac{\la_1}1 \le \frac{\la_2}2 \le \cdots \le \frac{\la_k}  k.
$$
Since these partitions are defined by linear inequalities, it follows
from Section~\ref{sec:cone} that their  weight \gf\ is rational, with a
denominator formed of cyclotomic polynomials. Still, there is no clear
reason to expect that~\cite{mbm-kimmo-lhp1}:
$$
\sum_{\la \ k-\hbox{\scriptsize{LHP}}} t^{|\la|}=\frac 1{(1-t)(1-t^3) \cdots (1-t^{2k-1})}.
$$
Several proofs have been given for this result and variations on
it. See for instance~\cite{mbm-kimmo-lhp2,corteel-savage-lee} and
references in the latter paper. Some of these proofs are based on a bijection
between lecture hall partitions 
and partitions into parts taken in $\{1,3, \ldots , 2k-1\}$, but these
bijections are never really simple~\cite{yee-lhp,eriksen-lhp}.

\subsubsection{$\ns$-rationality}
As we wrote in Section~\ref{sec:intuition-rational}, we do not know of a counting problem
that would yield a rational, but not $\ns$-rational series. It would
certainly be interesting to find one (even if it ruins some parts of
this paper).

Let us return  to Soittola's criterion for
$\ns$-rationality (Theorem~\ref{thm:soittola}). 
It is not always
easy to prove that a rational series has non-negative
coefficients. For instance, it was conjectured in~\cite{stanton-conj} that
for any odd $k$, the number of  partitions of $n$
into parts taken in $\{k, k+1, \ldots , 2k-1\}$ is a non-decreasing
function of $n$, for $n\ge 1$. In terms of \gfs, this means that the
series
$$
q+ \frac{1-q}{(1-q^k)(1-q^{k+1}) \cdots (1-q^{2k-1})}
$$
has non-negative coefficients. This was only proved
recently~\cite{stanton}. When $k$ is even, a similar result holds for the
series
$$
q+ \frac{1-q}{(1-q^k)(1-q^{k+1}) \cdots  (1-q^{2k})(1-q^{2k+1})}.
$$
Once the non-negativity of the coefficients has been established, it
is not hard to prove that these series are  $\ns$-rational. This
raises the question of finding a family of combinatorial objects that
they count.
\section{Algebraic generating functions}
\label{sec:alg}
\subsection{Definitions and properties}
The Laurent series $A(t)$ with coefficients in the field $R$ is said
to be  \emm algebraic,  (over $R(t)$) if it
satisfies a non-trivial \emm algebraic, equation:
$$
P(t,A(t))=0
$$
where $P$ is a bivariate polynomial with coefficients in $R$.

We assume below that $R=\qs$. 
Again, the set of algebraic  Laurent series possesses numerous interesting
properties~\cite[Ch.~6]{stanley-vol2},~\cite[Ch.~VII]{fla-sedg2}.
It is closed under sum, product, derivation,
reciprocals, but not under integration. These closure properties
become effective using 
either the theory of elimination or Gr\"obner bases, which are
implemented in most computer algebra packages. The
coefficients $a_n$ of an algebraic series $A(t)$ satisfy a linear
recurrence relation with polynomial coefficients: for $n$ large enough,
$$
p_0(n)a_n+p_1(n)a_{n-1}+p_2(n) a_{n-2} + \cdots + p_k(n)a_{n-k}=0.
$$
Thus the first $n$ coefficients can be computed using a linear
number of operations. 

There is no systematic way to express the coefficients of an algebraic
series in closed form. Still,  one can sometimes apply the
 \emm Lagrange inversion formula,:
\begin{proposition}
  Let $\Phi$ and $\Psi$ 
be two \fps\ and let  $U\equiv U(t)$ be the unique \fps\ with no
constant term satisfying 
$$
U=t\Phi(U).
$$
Then for $n>0$, the coefficient of $t^n$ in $\Psi(U)$ is:
$$[t^n]\Psi(U) = \frac 1{n} [t^{n-1}] \left(
\Psi'(t)\Phi(t)^n\right).$$
\end{proposition}
%
Given an algebraic  equation $P(t,A(t))=0$, one can decide whether
there exists a series $U(t)$ and  two \emm rational, series $\Phi$ and
$\Psi$ satisfying 
\beq\label{lagrange-eq}
U=t\Phi(U) \quad \hbox{ and } \quad A=\Psi(U).
\eeq
Indeed, such series
exist if and only if the \emm genus, of the curve $P(t,a)$ is
zero~\cite[Ch.~15]{abhyankar}. Moreover, both the genus and a
parametrization of the curve in 
the form~\eqref{lagrange-eq} can be determined algorithmically. 

\begin{example}[{\bf Finding a rational parametrization}]
The following algebraic equation was recently
obtained~\cite{mbm-jehanne}, after a highly 
non-combinatorial derivation, for the GF of  certain planar
graphs carrying a \emm hard-particle configuration,:
\begin{multline}
0=23328\,t^6A^4
+27\,t^4 \left( 91-2088\,t \right) A^3+t^2 \left( 86-3951\,t+46710\,t^2+3456\,t^3
 \right) A^2\\
+ \left( 1-69\,t+1598\,{t}^{2}-11743\,{t}^{3}-
14544\,{t}^{4} \right) A-1+66\,t-1495\,{t}^{2}+11485\,{t}^{3}+
128\,{t}^{4}.
\label{hard-part}
\end{multline}
 The package {\tt algcurves} of {\sc Maple}, and more precisely
the commands {\tt genus} and  {\tt parametrization}, reveal that a
 rational parametrization is obtained by setting
$$
t= -3\,{\frac { \left( 3\,U+7 \right)  \left( 9\,{U}^{2}+33\,U+37
 \right) }{ \left( 3\,U+1 \right) ^{4}}}.
$$
Of course, this is just the net result of {\sc Maple}, which is not
necessarily very meaningful for combinatorics. Still,
starting from this parametrization, one obtains after a few attempts 
an alternative parametrizing series $V$ with positive coefficients:
\beq\label{eq:param}
V=\frac t{(1-2V)(1-3V+3V^2)}.
\eeq
 The main interest of such a parametrization for this problem does
\emm not, lie 
in the possibility of applying the Lagrange inversion formula. Rather,
it suggests that a more combinatorial approach exists, based on the
enumeration of certain \emm trees,, in the vein
of~\cite{mbm-schaeffer-ising,bdg-pdures-bip}. It also 
gives a hint of what these trees may look like.
\end{example}

Another convenient tool borrowed from the theory of algebraic curves
 is the possibility to explore all branches of the curve $P(t,A(t))=0$
 in the neighbourhood of a given point $t_0$. This is based on Newton's
 polygon method. All branches have a \emm Puiseux expansion,, that is,
 an expansion of the  form:
$$
A(t) = \sum_{n\ge n_0} a_n (t-t_0)^{n/d}
$$
with $n_0 \in \zs$, $d\in \ps$. The coefficients $a_n$ belong to $\cs$
(in general, to an algebraic closure of the ground field). These
expansions can be computed automatically using standard software.  For
instance, the {\sc Maple} command {\tt puiseux} of the {\tt
  algcurves} package tells us that~\eqref{hard-part} has a unique solution that
is a \fps, the other three solutions starting with a term $t^{-2}$. 

Such Puiseux expansions are crucial for studying the \emm asymptotic
 behaviour, of the coefficients of an algebraic series $A(t)$. As in
 the rational case, one has  first
 to locate the singularities of $A(t)$, considered as a function of a
 complex variable $t$. These singularities are found among the roots of
 the \emm discriminant, and of the leading coefficient of $P(t,a)$
 (seen as a polynomial in $a$). The singular expansion of $A(t)$ near its
 singularities of smallest modulus can then be converted, using
 certain \emm transfer theorems,, into an asymptotic expansion of the
 coefficients~\cite[VII.4]{flajolet-odlyzko,fla-sedg2}. 

\begin{example}[{\bf Asymptotics of the coefficients of an algebraic series}]
Consider the series $V(t)$ defined by~\eqref{eq:param}. Its
 singularities lie among the roots of the discriminant
$$
\Delta(t)=-3+114t-4635t^2+55296t^3.
$$
Only one root is real. Denote it $t_0\sim 0.065$. The modulus of the
other two roots is smaller than $t_0$, so they could, in theory, be
candidates for singularities. However, $V(t)$ has non-negative
coefficients, and this implies, by Pringsheim's theorem, that one of
the roots of minimal modulus is real and positive. Hence $V(t)$ has a
unique singularity, lying at $t_0$. A Puiseux expansion at this point
gives
$$
V(t)= c_0- c_1\sqrt{1-t/t_0} + O(t-t_0),
$$
 for some explicit (positive) algebraic numbers $c_0$ and $c_1$, which
 translates  into 
$$
[t^n]V(t)= \frac{c_1}{2\sqrt \pi} t_0^{-n} n^{-3/2} \left(1+ o(1)\right).
$$
\end{example}
The determination of asymptotic expansions for the coefficients of
 algebraic series is probably not far from being completely automated,
 at least in the case of series with non-negative
 coefficients~\cite{chabaud,fla-sedg2}. The ``typical'' behaviour is 
\beq\label{asympt-alg}
a_n\sim \frac \kappa{\Gamma(d+1)} \rho ^{-n} {n^d},
\eeq
where $\kappa$ is an algebraic number and $d\in\qs\setminus\{-1,-2,-3
, \ldots\}$. Compare with the result~\eqref{asympt-rat} obtained for rational
series. Again,  the above statement is not exact, as the contribution
of \emm all, dominant singularities must be taken into
account. See~\cite[Thm.~VII.6]{fla-sedg2} for a complete statement.

\smallskip
Let us add that, again, one can guess if a series $A(t)$ given by its
first coefficients satisfies an algebraic equation $P(t,A(t))=0$ of a
given bi-degree $(d,e)$.  The guessing procedure requires to know at
least $(d+1)(e+1)$ coefficients, and amounts to solving a system of
linear equations. It is implemented in the package {\tt Gfun} of {\sc
  Maple}~\cite{gfun}.   For instance, given the 10 first coefficients of the series
$V(t)$ satisfying $V(0)=0$ and~\eqref{eq:param},
one automatically conjectures~\eqref{eq:param}.

\subsection{Plane trees}\label{sec:trees}
Our typical ``algebraic'' objects will be (plane) trees.
Let us begin with their usual intuitive recursive definition.
A \emm  tree, is a graph formed of a distinguished vertex
(called the \emm root,) to 
which are attached a certain number (possibly zero) of trees,
ordered from left to right. The number of these trees is the \emm
degree, of the root. The roots of these trees are the \emm
children, of the root. A more rigorous definition describes a tree
as a finite set of words on the alphabet $\ps$ satisfying certain
conditions~\cite{neveu}. We hope that our less formal definition and
Fig.~\ref{fig:tree-map}(a) suffice to understand what we mean. The
vertices of a tree are often called \emm nodes,. Nodes of 
degree 0 are called \emm leaves,, the others are called \emm inner, nodes.

\begin{figure}[t]
\begin{center}
\input{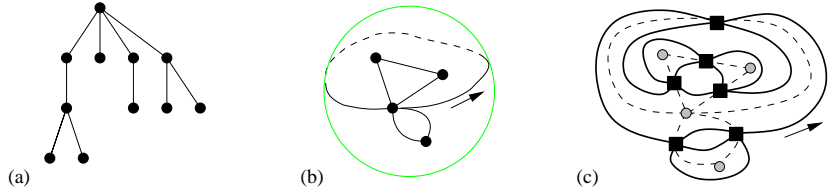}
\end{center}
\vskip -5mm\caption{(a) A plane tree. (b) A rooted planar map. (c) The
  corresponding 4-valent map (thick lines).}
\label{fig:tree-map}
\end{figure}

The enumeration of classes of trees  yields very often
algebraic equations. Let us consider  for instance the \emm
complete binary trees,, that is, the trees in which all vertices have
degree 0 or 2 (Fig.~\ref{fig-JFmap}). Let $a_n$ be the number of such trees having $n$
leaves. Then, by looking at the two (sub)trees of the root, one gets,
for $n>1$: 
$$a_n= \sum_{k=1}^{n-1} a_k a_{n-k}.
$$
The initial condition is $a_1=1$. In terms of GFs, this gives
$A(t)=t+A(t)^2,
$
which is easily solved:
\beq\label{eq:binary}
A(t)= \frac{1-\sqrt{1-4t}}2=\sum_{n\ge 0}\frac 1{n+1}{{2n}\choose n}
t^{n+1}.
\eeq

More generally, many algebraic series obtained in enumeration
are given as the first component of the solution of a  system of the form
\beq\label{system}
A_i= P_i(t, A_1, \ldots , A_k),
\eeq
for some polynomials $P_i(t,x_1, \ldots , x_k)$ having coefficients in
$\zs$. This system is 
said to be \emm proper, if $P_i$ has no constant term $(P_i(0,\ldots ,
0)=0)$ and does not contain any linear term $x_i$. It is \emm
positive, if the coefficients of the $P_i$ are non-negative. For instance,
$$
A_1=t^2+ A_1A_2 \quad \hbox{ and }\quad A_2=2tA_1^3
$$
is a proper positive system. The system is \emm quadratic, if every
 $P_i(t,x_1, \ldots , x_k)$ is a linear combination
 of the monomials $t$ and $x_\ell x_m$, for $1\le \ell\le m \le k$.
\begin{theorem}[\cite{stanley-vol2}, Thm. 6.6.10 and~\cite{salomaa},
    Thm. IV.2.2] \label{thm:system}
  A proper algebraic system has a unique solution $(A_1, \ldots ,
  A_k)$ in the set of \fps\ in $t$ with no constant term. This
  solution is called the \emm canonical, solution of the system. The
  series $A_1$ is also the first component of the solution of 

-- a proper quadratic  system,

-- a proper  system
  of the form $B_i=tQ_i(t,B_1, \ldots, B_\ell)$, for $1\le i\le \ell$.

\noindent These two systems can be chosen to be 
positive if the original system is positive. 
\end{theorem}
\begin{proof}
  Let us prove the last property, which we have not found in the above
  references. Assume $A_1$ satisfies~\eqref{system} and that this
  system is quadratic. The $i$th equation reads $A_i= m_it+n_i
  A_{\sigma(i)}A_{\tau(i)}  $. Rewrite each monomial $A_iA_j$
as  $tU_{ij}$   and add  the   equations 
$U_{ij }= t\left(m_im_j+m_in_jU_{\sigma(j)\tau(j)}
  +m_jn_iU_{\sigma(i)\tau(i)}
+ n_in_jU_{\sigma(i)\tau(i)} U_{\sigma(j)\tau(j)}\right)$. The new system
  has the required properties.
\end{proof}
\begin{definition}\label{def:N-algebraic}
  A series $A(t)$ is $\ns$-\emm algebraic, if it has coefficients in
  $\ns$ and if $A(t)-A(0)$ is the first component of the solution of a
  proper positive system.
\end{definition}
Proper positive  systems like~\eqref{system} can always be
given a combinatorial interpretation in terms of trees. Every vertex
of these trees carries a label $(i,c)$ where $i\in \llbracket k
\rrbracket$ and $c\in \ps$. We say that $i$ is the \emm type, of the
vertex and that $c$ is its \emm colour,. The \emm type, of a  tree is
the type of its root. Write $A_0=t$, so that $A_i=P_i(A_0,
A_1, \ldots , A_k)$. Let ${\cal A}_0$ be the set reduced to the 
tree with one node, labelled $(0,1)$. For $i\in \llbracket k
\rrbracket$, let ${\cal A}_i$ be the set of trees such that

-- the root has type $i$,

-- the types of the subtrees of the root, visited from left to right,
   are $0,\ldots , 0, 1, \ldots,$
$ 1,  \ldots , k, \ldots ,k$, in this    order,

-- if exactly $e_j$ children of the root have type $j$, the colour of the root is
   any integer in the interval $[1, m]$, where $m$ is the coefficient
   of $x_0^{e_0}\cdots x_k^{e_k}$ in $P_i(x_0, \ldots , x_k)$. 

Then it is not hard to see that $A_i(t)$ is the \gf\ of trees of type
$i$, counted by the number of leaves. This explains why trees will be,
in the rest of this paper, our typical ``algebraic'' objects.
\subsection{Context-free languages}
As in the case of rational (and, more precisely, $\ns$-rational)
series,  there exists a family of languages that is closely
related to algebraic series. 
A \emm context-free grammar, $G$ consists of

-- a set ${\cal S}= \{S_1, \ldots , S_k\}$ of \emm symbols,, with one
   distinguished symbol, say, $S_1$,

-- a finite alphabet $\cal A$ of \emm letters,, disjoint from ${\cal S}$,

-- a set of \emm rewriting rules, of the form $S_i \rightarrow w$
where $w$ is a non-empty word on the alphabet ${\cal S}\cup \cal A$.

The grammar is \emm  proper, if there is no rule $S_i \rightarrow S_j$.
The language ${\cal L}(G)$ generated by $G$ is the set of words \emm
on the alphabet, 
$\cal A$ that can be obtained from $S_1$ by applying iteratively the
rewriting rules. A language is \emm context-free, is there exists a
context-free grammar that generates it. In this case there exists also
a proper context-free grammar that generates it.

\begin{example}[{\bf Dyck words}]
 Consider the grammar $G$ having only one
symbol, $S$, alphabet $\{a,b\}$, and rules $S\rightarrow ab+ abS+ aSb+
aSbS$ (which is  short for 
$
S\rightarrow ab,  \quad S\rightarrow abS, \quad S\rightarrow aSb,
\quad S\rightarrow aSbS
$).
It is easy to see that ${\cal L}(G)$ is the set of non-empty words $u$
on
$\{a,b\}$ such that $|u|_a=|u|_b$ and for every prefix $v$ of $u$,
$|v|_a\ge|v|_b$. These words, called \emm Dyck words,,
 provide a simple encoding of the Dyck paths met in Example~\ref{ex:area}.
\end{example}

A \emm derivation tree, associated with $G$ is a plane tree in which
all inner nodes are labelled by symbols, and all leaves by letters, in
such a way that if a node is labelled $S_i$ and its children $w_1,
\ldots , w_k$ (from left to right), then the rewriting rule $S_i
\rightarrow w_1 \cdots w_k$ is in the grammar. If the root is labelled
$S_1$, then the word obtained
by reading the labels of the leaves in prefix order (\emm i.e.,, from
left to right) belongs to the
language generated by $G$. Conversely, 
for every word
$w$ in ${\cal L}(G)$, there exists at least one derivation tree with
root labelled $S_1$ that
gives $w$. The grammar is said to be \emm unambiguous, if every word
of ${\cal L}(G)$ admits a unique derivation tree.

Assume $G$ is proper. 
For $1\le i\le k$, let $A_i(t)$ be the \gf\ of derivation trees rooted
at $S_i$, counted by the number of leaves. With each rule $r$,
associate the monomial $M(r)=x_0^{e_0}\cdots x_k^{e_k}$ where $e_0$
[resp. $e_i$, with $i>0$]  is
the number of letters of $\cal A$ [resp. occurrences of $S_i$] in the
right-hand side of $r$. Then the series $A_1, \ldots , A_k$ form the
canonical solution of the proper positive system~\eqref{system},
with
$$
P_i(x_0,x_1, \ldots , x_k) = \sum_r M(r),
$$
where the sum runs over all rules $r$ with left-hand side $S_i$.

Conversely, starting from a positive  system $B_i=tQ_i(t, B_1, \ldots,
B_k)$ and its canonical solution, it is
always possible to construct an unambiguous grammar with symbols $S_1,
\ldots , S_k$ such that $B_i$  is the \gf\ of derivation 
trees rooted at $S_i$ (the  idea is to introduce a new letter
$a_i$ for each occurrence of $t$). In view of Theorem~\ref{thm:system} and
Definition~\ref{def:N-algebraic}, this gives the following alternative
characterization of $\ns$-algebraic series:

\begin{proposition}
  A series $A(t)$ 
is $\ns$-algebraic if
  and only if  only $A(0)\in\ns$ and there exists an unambiguous
  context-free language   having \gf\ $A(t)-A(0)$.
\end{proposition}
\subsection{The combinatorial intuition of algebraic \gfs}
We have described two families of combinatorial objects that
naturally yield algebraic GFs: plane trees and words of
unambiguous context-free languages. We have, moreover, shown a close
relationship between  these two 
types of objects.  These two families
convey the standard intuition of what a family with an algebraic \gf\
looks like: the algebraicity suggests that it \emm may, (or should...)
be possible to give a recursive 
description of the objects based on disjoint union of sets and
concatenation of objects. Underlying such a description is a
context-free grammar. This intuition is the basis of the so-called
Sch\"utzenberger methodology, according to which the ``right''
combinatorial way of proving algebraicity is to describe a bijection
between the objects one counts and the words of an unambiguous
context-free language. This approach has led in the 80's and 90's
to numerous satisfactory explanations of the algebraicity of certain
series, and we describe some of them in this subsection.
Let us, however, warn the reader that the similarities with the
rational case will stop here. Indeed, it seems that the ``context-free''
intuition is far from explaining all algebraicity phenomena in
enumerative combinatorics. In particular,
\begin{itemize}
  \item[$(i)$] it is very likely that many families of objects have an
  algebraic, but not $\ns$-algebraic  \gf ,
\item[$(ii)$] there are many  families of combinatorial
  objects with an algebraic GF that do not exhibit a clear
  ``context-free'' structure, based on union and concatenation. For
  several of these families, there is just no explanation of this type,
  be it clear or not.
\end{itemize}
 This will be discussed in the next subsections.
For the moment, let us illustrate the ``context-free'' intuition.

\subsubsection{Walks on a line}\label{sec:line}
Let $\cS$ be a finite subset of $\zs$.  Let $\cW$ be the set of walks
on the line $\zs$ that start from $0$ and take their steps in
$\cS$. The \emm length, of a walk is its number of steps. Let $\cW_k$
be the set of walks ending at position $k$. For $k\ge 0$,  let $\cM_k$
be the subset of  $\cW_k$ consisting of walks that never visit a
negative position, and let $\cM$ be the union of the sets $\cM_k$. In
probabilistic terms, the walks in $\cM$ would be called \emm meanders,
and the walks of $\cM_0$ \emm excursions., Of course, a walk is simply
a sequence of steps, hence a word on the alphabet $\cS$. Thus the sets
of walks we have defined can be considered as languages on this
alphabet.

\begin{theorem}\label{meanders}
  The language $\cW$ is simply $\cS^*$ and is thus regular.  The
  languages $\cM$, $\cW_k$ and $\cM_k$ are unambiguous context-free
   for all $k$. 
\end{theorem}
\begin{proof}
We only describe the (very simple)  case $\cS=\{+1,-1\}$, to
illustrate the ideas that are involved in the construction of the
grammar. 
We encode the steps $+1$ by the letter $a$, the steps
  $-1$ by $b$, and introduce some auxiliary languages:

$\bullet$ $\cM_0^-$, the subset of $\cW_0$ formed of walks that
never visit a positive position,


$\bullet$  $\cW_0^+$ [resp.  $\cW_0^-$], the subset of $\cW_0$ formed of walks that start  with  $a$ [resp.  $b$].


\noindent The language $\cM_0$ will be generated from the symbol
$M_0$, and similarly for the other languages.
By looking at the first time a walk of $\cM_0$ [resp. $\cM_0^-$]
reaches position $ 0$ after its first step, one obtains 
$$
M_0 \rightarrow a(1+M_0)b  (1+M_0) \quad \hbox{ and } \quad
M_0^- \rightarrow b(1+M_0^-)a   (1+M_0^-).
$$
 By considering the last visit to 0 of a walk of $\cM_k$, one obtains,
 for $k>0$: 
$$
M_k\rightarrow   (1+M_0) a \left( 1_{k=1} +M_{k-1}\right).
$$
This is easily adapted to general meanders:
$$
M\rightarrow  M_0+(1+M_0) a \left( 1 + M\right).
$$
Considering the first step of a walk of $\cW_0$ gives
$$ 
W_0\rightarrow W_0^++ W_0^- \quad \hbox{ with }
\quad W_0^+\rightarrow M_0(1+W_0^-) \quad \hbox{ and }  \quad
W_0^-\rightarrow M_0^-(1+W_0^+).
$$
Finally, for $k>0$, looking at the first visit at 1 [resp. $-1$]  of a
walk of $\cW_k$ [resp. $\cW_{-k}$] yields
$$
W_k\rightarrow(1+M_0^-)a \left( 1_{k=1} + W_{k-1}\right)  \quad\left[\hbox
  {resp.} \quad W_{-k}\rightarrow(1+M_0)b \left( 1_{k=1} + W_{-(k-1)}\right)\right].
$$

For a general set of steps $\cS$, various grammars have been described
for the languages $\cM_k$  
of meanders~\cite{duchon,labelle-yeh,labelle}. For $\cW_k$, we refer
to~\cite[Section 4]{labelle} where the (representative) case
$\cS=\{-2,-1,0,1,2\}$ is treated. 
\end{proof}

Theorem~\ref{meanders} is often described in terms of walks in $\zs^2$
starting from $(0,0)$ and taking their
steps in $\{(1,j), j \in \cS\}$. The conditions on the
positions of the walks that lead to the definition of $\cM_k$ and
$\cW_k$ are restated in terms of conditions on the \emm ordinates, of
the vertices visited by the walk. A harmless generalization is
obtained by taking steps 
in a finite subset $\cS$ of $\ps\times\zs$.
A walk is still encoded by
a word on the alphabet $\cS$. The languages $\cW_k$ remain unambiguous
context-free. If each step $(i,j)$ is, moreover, weighted by
a rational number
 $w_{i,j}$, then the  \gf\ of
walks of $\cW$, counted by  the coordinates of their endpoint, is
$$
W(t,s)=\frac 1{1-\sum_{(i,j)\in \cS} w_{i,j}t^i s^j}.
$$
The \gf\ $W_k(t)$ that counts (weighted) walks ending at ordinate $k$ is the
coefficient of $s^k$ in $W(t,s)$. Since $\cW_k$ is unambiguous
 context-free, the series $W_k(t)$ is algebraic. 
This
gives a combinatorial explanation of the following
result~\cite[Thm. 6.3.3]{stanley-vol2}. 

\begin{theorem}[{\bf Diagonals of rational series}] \label{thm:diag}
  Let $A(x,y)=\sum_{m,n\ge 0} a_{m,n}x^m y^n$ be a series in two
  variables $x$ and $y$, with   coefficients in $\qs$, that is
  rational.
Then the diagonal of $A$, that is, the series $\Delta A(t)=\sum_{n\ge 0}
  a_{n,n}t^n$, is algebraic.
\end{theorem}
\begin{proof}
  By linearity, it suffices to consider the case
$$
A(x,y)= \frac{x^ay^b}{1-\sum_{0\le m,n \le d} c_{m,n} x^m y^n},
$$
with $c_{0,0}=0$. Set $x=ts$ and $y=t/s$. The diagonal of $A$
satisfies
$$
\Delta A(t^2)= [s^0] A(ts,t/s)=
t^{a+b} [s^{b-a}] \frac 1 {1-\sum_{0\le m,n \le d}
  c_{m,n} t^{m+n} s^{m-n}},
$$
which is algebraic as it counts weighted paths in $\cW_{b-a}$, for a
certain set of steps. Hence $\Delta A(t)$ is algebraic too.
\end{proof}

 The converse of Theorem~\ref{thm:diag} holds: every series
$B(t)$ that is algebraic over $\qs(t)$ is the diagonal of a bivariate
rational series $A(x,t)$~\cite{safonov}. 

\smallskip
\noindent{\bf Note.} If one is simply interested in obtaining a set of
algebraic equations defining the GFs of the sets $\cM_k$ and $\cW_k$,
a more straightforward approach is to use a partial fraction
decomposition (for $\cW_k$) and the kernel method (for
$\cM_k$). See~\cite[6.3]{stanley-vol2}, and \cite[Example 3]{bousquet-petkovsek-1}.

\subsubsection{Directed animals}\label{sec:diriges}
Let us move to an example where a neat context-free exists, but is
uneasy to discover. We return to the \emm directed animals, defined in
Section~\ref{sec:predicting}. As discussed there, there is no simple
explanation as to why the number of compact-source animals is so
simple (Theorem~\ref{thm:compact}). Still, there is a convincing
explanation for the \emm algebraicity, of the corresponding series:
directed animals have, indeed, a context-free structure.
This structure was discovered  a few years after the proof of
Theorem~\ref{thm:compact}, with 
the development by Viennot of the theory of {\em
  heaps}~\cite{viennot-heaps}, a geometric version of partially
commutative monoids 
\cite{cartier-foata}. 
 Intuitively, a heap is obtained by dropping vertically some
solid pieces, the one after the other. Thus, a piece lies either on
the ``floor'' 
(then it is said to be {\em minimal}), or covers, at
least partially, another piece.

Directed animals {\em are}, in essence, heaps. To see this,
replace every point of the animal 
by a  {\em dimer} (Fig.~\ref{bijection}). 
Note that if the animal has a unique source, the associated heap has
 a unique minimal piece. Such heaps are named {\em pyramids}.

\begin{figure}[hbt] \begin{center} 
\input{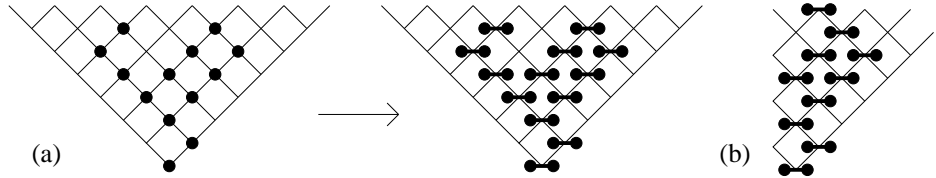}
\end{center} 
\vskip -5mm\caption{(a) A directed animal and the associated pyramid. 
(b) A half-pyramid.}
\label{bijection} 
\end{figure}

What makes heaps  interesting here is that there exists a {\em monoid
  structure} on the set of heaps: 
The product of two heaps is obtained by putting one heap above the
other and dropping its pieces.  
This product is the key in our  context-free 
description of directed animals.


Let us begin with the description of pyramids (one-source
animals). A pyramid is either a 
{\em half-pyramid} (Fig. \ref{bijection}(b)), 
or the product of a half-pyramid and  a pyramid (Fig.
\ref{pyramide}, top). Let $P(t)$ denote the GF of pyramids counted by
the number of dimers, and $H(t)$ denote the GF of
half-pyramids. Then $P(t)=H(t)(1+P(t))$.  
Now, a half-pyramid may be reduced to a single 
dimer. If it has several dimers, it is the product of a single dimer
and of one or two 
half-pyramids (Fig. \ref{pyramide}, bottom), which implies
$H(t)=t+tH(t)+tH^2(t)$.  

\begin{figure}[hbt] \begin{center} 
\input{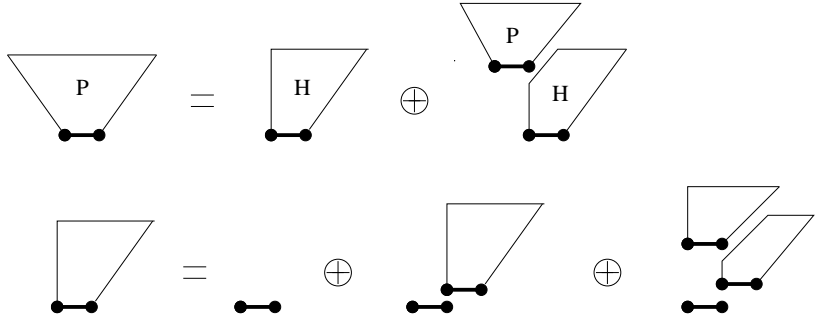}
\caption{Decomposition of pyramids (top) and half-pyramids (bottom).}
\label{pyramide} \end{center} \end{figure}

\noindent A trivial computation finally
provides the GF of directed (single-source) animals:
$$P(t)=\frac{1}{2}\left( \sqrt{\frac{1+t}{1-3t}} -1 \right) \quad \quad 
\left(\hbox{ while } 
H(t)= \frac{1-t-\sqrt{(1+t)(1-3t)}}{2t}\right).$$
The enumeration of compact-source directed animals is equivalent to the enumeration of
heaps having a compact {\em basis} (the minimal dimers are
adjacent). The generating function of heaps having a compact basis 
formed with $k$ dimers is $P(t)H(t)^{k-1}$ (Fig. \ref{basecompacte}), which implies that the generating function
of compact-source animals is $$\frac{P(t)}{1-H(t)}=\frac{t}{1-3t}.$$

\begin{figure}[hbt] \begin{center} 
\input{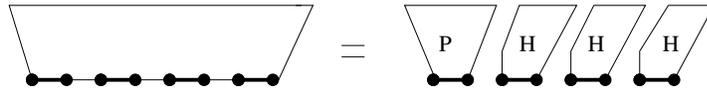}
\caption{Decomposition of heaps having a compact basis.}
\label{basecompacte} \end{center} \end{figure}


\subsection{The world of planar maps}\label{sec:world}
We have seen in Section~\ref{sec:trees} that plane trees are the paradigm for
objects with an algebraic \gf. A more general family of plane
objects seems to be just as deeply associated with algebraic series,
but for reasons that are far more mysterious: \emm planar maps,.

A (planar) map is a proper embedding of a planar graph in the
sphere (Fig.~\ref{fig:tree-map}(b)). In order to avoid symmetries, all
the maps we consider 
are \emm rooted,: this means that one edge is distinguished and
oriented. Maps are only considered up to a continuous deformation of the
sphere. A map induces a 2-cell decomposition of the sphere: the
cells of dimension 0 [resp. 1, 2] are called \emm vertices, [resp. \emm
  edges,, \emm faces,]. Hence plane trees are maps with a single face.

The interest for the enumeration of planar maps dates back to the
early 60's, in connection with the 4-colour theorem. The first results
are due to
Tutte~\cite{tutte-triangulations,tutte-slicings,tutte-maps}. Ten to
fifteen years later, maps started to 
be investigated independently 
in theoretical physics, as a model for 2-dimensional
\emm quantum gravity,~\cite{BIPZ,BIZ}. However, neither the recursive approach
used by Tutte and his disciples, nor the physics approach based on
 matrix integrals were able to explain in a combinatorially
satisfactory way the following observations:

-- the \gfs\ of many classes of planar maps are
   algebraic,

-- the associated numbers are often irritatingly simple.

\noindent Let us illustrate this with three examples:

\noindent
{\bf 1. General maps.} The number of planar maps having $n$ edges is~\cite{Tutte-planar}:
\beq\label{map:expr}
g_n= \frac{2. 3^n}{(n+1)(n+2)}{{2n}\choose n}.
\eeq
The associated \gf \ $G\equiv G(t)=\sum_{n\ge 0 }g_n t^n$ satisfies:
\beq\label{map:alg}
-1+16t+(1-18t)G+27t^2G^2=0.
\eeq

\noindent
{\bf 2. Loopless triangulations.} The number of loopless \emm  triangulations, (maps in which all faces have degree 3)  
having $2n+2$ faces is~\cite{mullin-triangulations}:
$$ 
t_n= { \frac{2^n }{(n+1)(2n+1)}
{{3n}\choose n}}.  $$
The associated \gf\ $T\equiv T(t)=\sum_n t_n t^n$ satisfies
$$
 1-27t+(-1+36t)T-8tT^2-16t^2T^3=0.
$$

\noindent
{\bf 3. Three-connected triangulations.} 
The number of 3-connected triangulations
having $2n+2$ faces is~\cite{tutte-triangulations}:
$$ 
m_n= { \frac{2 }{(n+1)(3n+2)}
{{4n+1}\choose n}}.  $$
  The associated \gf\ $M\equiv M(t)=\sum_n t_n t^n$ satisfies
$$
-1+16t+(1-20t)M+(3t+8t^2)M^2+3t^2M^3+t^3M^4=0.
$$
These maps are in bijection with rooted \emm maximal, planar simple graphs
(graphs with no loop nor multiple edge that lose planarity as soon
as one adds an edge).

\smallskip
At last, in  the past ten years, a general combinatorial picture has
emerged, suggesting that maps are, in essence, unrooted plane trees.
In what follows, we illustrate on the
example of general maps  the main three approaches that now exist, and
give references for further  developments of these methods.

\subsubsection{The recursive approach} We leave to the reader to
experience personally that maps do not have an obvious context-free
structure. Still, maps \emm do, have a simple recursive structure,
based on the deletion of the root-edge. However, in order to exploit this
structure,  one is forced to keep track of the degree of the
\emm root-face, (the face lying to the right of the root edge). The decomposition illustrated in Fig.~\ref{fig:map-dec} leads
in a few lines to the following equation:
\beq\label{eq:maps}
G(u,t)= 1+tu^2 G(u,t)^2 + tu\ \frac{uG(u,t)-G(1,t)}{u-1},
\eeq
where $G(u,t)$ counts planar maps by the number of edges ($t$) and the
degree of the root-face ($u$).

\begin{figure}[t]
\begin{center}
\input{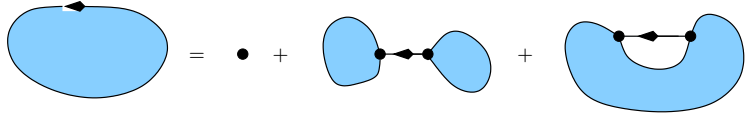}
\end{center}
\vskip -5mm\caption{Tutte's decomposition of rooted planar maps.}
\label{fig:map-dec}
\end{figure}

It can be checked that the above equation defines $G(u,t)$ uniquely as
a \fps\ in $t$ (with polynomial coefficients in $u$). However, it is
not clear on the equation why $G(1,t)$ (and hence $G(u,t)$) are
algebraic. In his original paper, Tutte first guessed the value of
$G_1(t):=G(1,t)$, and then proved the existence of a series $G(u,t)$ that fits
with $G_1(t)$ when $u=1$, and satisfies the above equation. Still, a
bit later, Brown came with a \emm method, for solving~\eqref{eq:maps}:
the so-called \emm quadratic
method,~\cite[Sec.~2.9]{brown-square,goulden-jackson}.
Write~\eqref{eq:maps} in the form
$(2aG(u,t)+b)^2=\delta,
$
where $a,b$ and $\delta$ are polynomials in $t, u$ and $G_1(t)$. That
is,
$$
\left( 2tu^2(u-1)G(u,t)+tu^2-u+1\right)^2= 
4t^2u^3(u-1)G_1+(1-u)^2-4tu^4+6tu^3+u^4t^2-2tu^2.
$$
It is not hard to see, even without knowing the value of $G(u,t)$,
that there exists a (unique) \fps\ in $t$, say $U\equiv U(t)$, that cancels
the left-hand side of this equation. That is,
$$
U=1+tU^2 +2tU^2(U-1)G(U,t).
$$
This implies that the series $U$ is a \emm double root, of the
polynomial $\delta$ that lies on the right-hand side. The
discriminant of this polynomial (in $u$) thus vanishes: this gives the
algebraic equation~\eqref{map:alg} satisfied by $G(1,t)$.

\smallskip
The enumeration of many other families of planar maps can also be
attacked by a recursive description based on the deletion of an edge
(or vertex, or face...). See for instance~\cite{mullin-triangulations} for
$2$-connected 
triangulations, or~\cite{bender-canfield} for maps with prescribed face
degrees.  (For maps with high connectivity, like 3-connected
triangulations, an additional \emm composition formula, is often
required~\cite{tutte-triangulations,banderier-schaeffer}.) The
resulting equations are usually of the form
\beq\label{cat}
P(F(u),F_1, \ldots , F_k,t,u)=0,
\eeq
where $F(u)\equiv F(t,u)$, the main \gf, is a series in $t$ with
polynomial coefficients in $u$, and $F_1, \ldots , F_k$ are series in
$t$ only, independent of $u$. Brown's \emm quadratic method, applies
as long as the degree in $F(u)$ is 2 (for the linear case, see the
\emm kernel method,
in~\cite{bousquet-petkovsek-1,hexacephale}). Recently, it was
understood how 
these equations could be solved in full generality~\cite{mbm-jehanne}.
Moreover, the solution of any (well-founded) equation of the above type
was shown to be algebraic. This provides two types
of enumerative results:

-- the proof that many map  \gfs\ are algebraic: it now suffices to
   exhibit an equation of the form~\eqref{cat}, or to explain why such
   an equation \emm exists,,

-- the solution of previously unsolved map problems (like the
   enumeration of hard-particle configurations on maps, which led
   to~\eqref{hard-part}, or that of triangulations with high vertex
   degrees~\cite{bernardi}). 

\subsubsection{Matrix integrals}
In the late 70's, it was understood by a group of physicists that
certain matrix integral techniques coming from quantum field
theory could be used to attack enumerative problems on maps~\cite{BIPZ,BIZ}.
This approach proved to be extremely efficient (even if it is usually
not fully rigorous). The first step is fairly automatized, and consists
in converting the description of maps into a certain integral. For
instance, the relevant integral for the enumeration of 4-valent maps
(maps in which all vertices have degree 4) is
$$
Z(t,N)=\frac {2^{N(N-1)/2}}{(2\pi)^{N^2/2}}\int dH \displaystyle
e^{\displaystyle {\rm tr}\left(
- {H^2}/ 2+  t H^4/N\right)},$$
where the integration space is that of hermitian matrices
$H$ of size $N$, equipped with the Lebesgue measure
$
dH=\prod dx_{kk} \prod_{k<\ell} dx_{k\ell}dy_{k\ell}
$
with $h_{k\ell}=x_{k\ell}+iy_{k\ell}$. As there is a classical
bijection between 4-valent maps with $n$ vertices and planar maps with
$n$ edges (Fig.~\ref{fig:tree-map}(c)), we are still dealing with our
reference problem: the 
enumeration of general planar maps. The  connection between the above
integral and maps is
$$
G(t)= t E'(t) \quad \hbox{ with } \quad E(t)=\lim _{N\rightarrow
  \infty} \frac 1 {N^2} \log Z(t,N).
$$
Other map problems lead to integrals involving  several hermitian
matrices~\cite{Ka86}. We refer to~\cite{Sacha} for a neat explanation
of the encoding 
of map problems by integrals, and to~\cite{difrancesco-survol,eynard} (and
references therein) 
for the evaluation of integrals. 

\subsubsection{Planar maps and trees}

\begin{figure}[b]
\begin{center}
\input{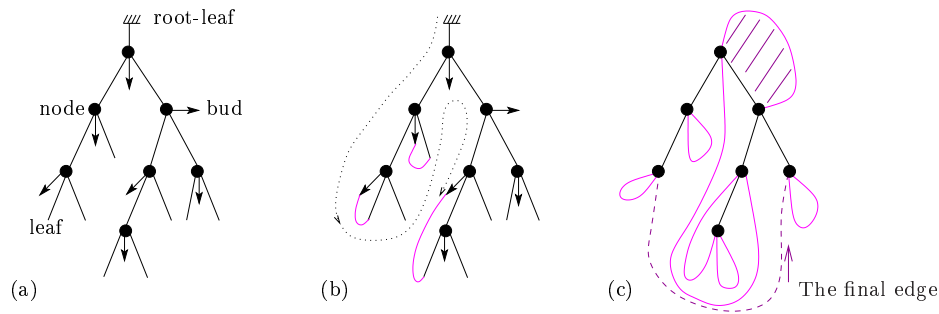}
\end{center}
\vskip -5mm\caption{(a). A budding tree. (b) An intermediate step in the matching
  procedure. (c) The resulting 4-valent map, with its marked face.}
\label{fig:blossoming}
\end{figure}
We finally come to a combinatorial explanation of the formula/equation
for $g_n$ and $G(t)$.
Take a plane binary tree with $n$ (inner) nodes, planted at a leaf,
and add to every inner node a new 
distinguished child, called a \emm bud,. At each node, we have three
choices for the position of the bud (Fig.~\ref{fig:blossoming}(a)). The
new tree, called \emm budding tree,, has now $n$ buds
and $n+2$ leaves. 
Now start from the root and walk around the tree in counterclockwise
order, paying attention to the sequence of buds and leaves you
meet. Each time a bud is immediately followed by a leaf in this
sequence, match them by forming a new edge (Fig.~\ref{fig:blossoming}(b)) and
then go on walking around the plane figure thus obtained. At the end,
exactly two leaves remain unmatched. Match them together and orient
this final edge in one of the two possible ways. Also, mark the face to
the left of the matching edge that ends at the root-leaf.

\begin{theorem}[\cite{Sch97}]
  The above correspondence is a bijection between pairs $(T,\eps)$
  where $T$ is a   budding
  tree having $n$ inner nodes and $\eps\in\{0,1\}$, and $4$-valent
  maps with $n$ vertices and a marked face.
\end{theorem}
The value of $\eps$ tells  how to orient the final matching edge.
Schaeffer first used this bijection to explain combinatorially
the formula~\eqref{map:expr}. Indeed, the number of   budding
  trees with $n$ inner nodes is clearly $3^n {{2n}\choose n} /(n+1)$
  (see~\eqref{eq:binary}), 
  while the number of 4-valent maps with $n$ vertices and a marked
  face is $(n+2)g_n$. Eq.~\eqref{map:expr} follows.

 Later, it was realized
  that this construction could also be used to explain the
  algebraicity of the series $G(t)$~\cite{BDG-planaires}.
Say that a budding tree is \emm balanced, if the root-leaf is not
matched by a bud. Take such a tree, match all buds, and orient the
final edge from the root-leaf to the other unmatched leaf. This gives
 a bijection between balanced budding trees and 4-valent maps. We
thus have to count balanced trees, or, equivalently, the unbalanced
ones. By re-rooting them at the bud that matches the root-leaf, one sees
that they are in bijection with a node attached to three budding trees. This
gives
$$
G(t)=B(t)-tB(t)^3, \quad \hbox{ where }\quad B(t)= 3t(1+B(t))^2
$$
counts budding trees by (inner) nodes.
The above construction involves taking a \emm difference, of 
$\ns$-algebraic series, which needs not be
$\ns$-algebraic. We actually conjecture that the series $G(t)$
is not $\ns$-algebraic (see Section~\ref{sec:N-alg}).

\smallskip
There is little doubt that the above construction (once described
in greater detail...) explains in  a very
satisfactory way both the simplicity of the formula giving $g_n$ and
the algebraicity of $G(t)$. Moreover, this  is not an \emm ad
hoc,, isolated magic trick: over the past ten years, it was realized
that this construction is one in a family of constructions  of the
same type, which
apply to numerous families of maps (Eulerian maps~\cite{Sch97}, maps with
prescribed vertex degrees~\cite{BDG-planaires},
constellations~\cite{mbm-schaeffer-constellations}, bipartite maps
with prescribed 
degrees~\cite{mbm-schaeffer-ising}, maps with higher connectivity~\cite{poulalhon-3-connected,fusy-3-connected}). Definitely, these
constructions reveal a lot about the combinatorial nature of
planar maps. 

To conclude this section, let us mention that a different
combinatorial construction for general planar maps, 
discovered 
in the early
80's~\cite{cori-vauquelin}, has recently been
simplified~\cite{chassaing-schaeffer} and 
adapted to other families of
maps~\cite{dellungo,jacquard-schaeffer,bdg-labelled,bdg-mobiles}. It
is a bit less easy to handle 
than the one based on trees with buds, but it allows one to keep track
of the distances between some vertices of the map. This has led to
remarkable connections with a random probability distribution called
the \emm Integrated SuperBrownian
Excursion,~\cite{chassaing-schaeffer}. A third type of construction
has emerged even more recently~\cite{bernardi-kreweras} for 2-connected
triangulations, but no ones knows at the moment whether it will remain
isolated or is just the tip of another iceberg.

\subsection{Algebraic series: some questions}
\label{section-strange}
We begin with three simple classes of objects that have an algebraic GF,
but for reasons that remain mysterious.
We then discuss a possible criterion (or necessary condition) for
$\ns$-algebraicity, and finally the algebraicity of certain
hypergeometric series.

\subsubsection{Kreweras' words and walks on the quarter plane}
\label{section-quarter}
\noindent
Let $\cL$ be the set of words $u$ on the alphabet $\{a,b,c\}$ such
  that for every prefix $v$ of $u$,  $|v|_a\ge |v|_b$ and 
  $|v|_a\ge |v|_c$. 
These words encode certain walks on the plane: these walks start
at  $(0,0)$, are made of three types of steps, $a=(1,1)$, $b=(-1,0)$ and
  $c=(0,-1)$, and never leave the first quadrant of the plane, defined
  by $x,y\ge 0$. The \emm pumping lemma,~\cite[Thm.~4.7]{hopcroft},
  applied to the word $a^nb^nc^n$, 
  shows that the language $\cL$ is not context-free. However, its 
  \gf\ \emm is, algebraic. More precisely, let us denote by
  $\ell_{i,j}(n)$ the number of words 
$u$ of $\cL$ of length $n$ such that $|u|_a-|u|_b=i$ and
$|u|_a-|u|_c=j$. They correspond to walks of length $n$ ending at position
  $(i,j)$. Then  the  associated three-variable \gf\ is 
$$
L(u,v;t)= \sum_{i,j, n} \ell_{i,j}(n)u^i v^j t^n= \frac{\left( 1/W-\bu\right) \sqrt{1-uW^2}
+\left( 1/W-\bv\right) \sqrt{1-vW^2}
}{uv-t(u +v+u^2v^2)}
- \frac 1{uvt}
$$
where $\bu=1/u$, $\bv = 1/v$ and  $W\equiv W(t)$ is the unique
power series in $t$ satisfying 
$
W=t(2+W^3).
$
Moreover, the number of  walks ending at $(i,0)$ is remarkably simple:
$$
\ell_{i,0}(3n+2i)=\frac{4 ^n (2i+1)}{(n+i+1)(2n+2i+1)} {{2i} \choose i}
{{3n+2i} \choose n} .
$$
The latter formula was proved in 1965 by Kreweras, in a fairly
complicated way~\cite{kreweras}. This rather mysterious result has
attracted the 
attention of several combinatorialists since its
publication~\cite{bousquet-kreweras,gessel-proba,niederhausen-ballot}.
The first combinatorial explanation of the above formula (in the case
$i=0$) has just been found by Bernardi~\cite{bernardi-kreweras}. 


\smallskip
Walks in the quarter plane do not always have
an algebraic GF: for instance, the number of \emm square lattice walks,
(with North, South, East and West steps) of size $2n$ that start and
end at $(0,0)$ and remain in the quarter plane is
$$
\frac{1}{(2n+1)(2n+4)} {{2n+2} \choose {n+1}}^2 \sim
\frac{4^{2n+1}}{\pi n^3},
$$
and this asymptotic behaviour prevents the corresponding \gf\ from
being algebraic (see~\eqref{asympt-alg}). The above formula is easily
proved by looking at the 
projections of the walk onto the horizontal and vertical axes.

\subsubsection{Walks on the slit plane}
\label{section-slit}
\noindent
Take  now \emm any,
finite set of steps $\cS \subset \zs\times \{-1,0,1\}$ (we say that these
steps have \emm small height variations,).  Let $s_{i,j}(n)$ be the number
of walks of length $n$ that
start from the origin,  consist of steps of $\cS$, never return to the
non-positive horizontal 
axis $\{(-k,0), k\ge 0\}$,  and  end at $(i,j)$.
 Let $S(u,v;t)$ be the associated \gf :
$$
S(u,v;t)= \sum_{i,j \in \zs, n \ge 0} s_{i,j}(n)u^i v^j t^n.
$$
Then this series is \emm always, algebraic, as well as the series 
$S_{i,j}(t):= \sum_n s_{i,j}(n) t^n$ that counts walks ending at
$(i,j)$~\cite{bousquet-slit,bousquet-schaeffer-slit}.
For instance, when $\cS$ is formed of the usual
square lattice steps (North, South, West and East), then
$$
{S(u,v;t)= \frac{\left( 1-2t(1+\bu)+\sqrt{1-4t}\right)^{1/2}
\left( 1+2t(1-\bu)+\sqrt{1+4t}\right)^{1/2}}{1-t(u+\bu +v+\bv)}}
$$
with $\bu =1/u$ and $\bv=1/v$.
Moreover, the number of walks  ending at certain
specific points is remarkably simple. For instance:
$$
s_{1,0}(2n+1)= C_{2n+1} , \quad
s_{0,1}(2n+1)= 4^n C_n, \quad
s_{-1,1}( 2n)= C_{2n},
$$
where $C_n= {{2n} \choose n}/(n+1)$ is the  $n$th Catalan
number, which counts binary trees~\eqref{eq:binary}, Dyck words, and numerous other
combinatorial objects~\cite[Ch.~6]{stanley-vol2}. The first of these
three identities has been proved combinatorially 
\cite{florence}. The others still defeat our understanding.

\subsubsection{\bf Embedded binary trees}
We consider again the complete binary trees met at the beginning of
Section~\ref{sec:trees}. Let us associate with each (inner) node 
of such a tree a label, equal to the 
difference between the number of right steps and the number of left
steps one takes when
 going from the root to the node. In other words, the label of the
node is its abscissa in the natural integer embedding of the tree
(Fig.~\ref{fig-JFmap}).

\begin{figure}[hbt]
\begin{center}
\input{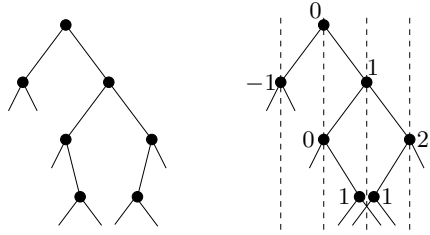}
\end{center}
\vskip -5mm\caption{The integer embedding of a binary tree.}
 \label{fig-JFmap}
\end{figure}

\noindent  Let $S_j\equiv S_j(t,u)$ be the \gf \ of binary trees counted by the
number of nodes (variable $t$) and the number of nodes at
abscissa $j$ (variable $u$).
Then for all $j \in \zs$, this series is
 algebraic of degree (at most) $8$
(while $S_j(t,1)$ is quadratic)~\cite{mbm-ise}.
Moreover,  for  $j \ge 0$,
$$
S_j= T\,\frac{(1+\mu Z^j)(1+\mu Z^{j+5})}{(1+\mu Z^{j+2})(1+\mu Z^{j+3})},
$$
where $$
T=1+tT^2, \quad \quad 
Z=t\, {\frac{ \left( 1+{Z}^{2} \right) ^{2
}} { \left( 1-Z+{Z}^{2} \right) }},
$$
and $\mu\equiv \mu(t,u)$ is the unique \fps\ in $t$
satisfying
$$
\mu= (u-1) \frac{Z(1+\mu Z)^2(1+\mu Z^2)(1+\mu Z^6)}
{(1+Z)^2(1+Z+Z^2)(1-Z)^3(1-\mu^2 Z^5)}.
$$

Why is that so? This algebraicity property holds as well for other
families of labelled trees~\cite{mbm-ise,BDG-trees}. 
From these series, one can derive certain limit results on the
 distribution of the number of nodes at abscissa $\lfloor \lambda n
 ^{1/4}\rfloor $ in a random tree with $n$ nodes~\cite{mbm-ise}. These
 results provide some  information  about the law of  the
 \emm integrated  super-Brownian excursion,~\cite{mbm-ise,mbm-janson}.

\subsubsection{$\ns$-algebraicity}\label{sec:N-alg}
$\ns$-algebraic series have been defined in Section~\ref{sec:trees} in terms of
positive proper algebraic systems. The author has been unable to find
in the literature a criterion, or even a necessary condition for an
algebraic series with coefficients in $\ns$ to be $\ns$-algebraic. Nor even  an algebraic series with coefficients in
$\ns$ that would \emm not, be $\ns$-algebraic (together with a proof of this
statement...). 

A partial answer could be provided by the study of the possible
asymptotic behaviour of coefficients of $\ns$-algebraic series. It is
very likely that not all behaviours of the form~\eqref{asympt-alg} are
possible. An important result in this direction states that, if a
proper positive system~\eqref{system} is \emm strongly connected,, the $n$th
coefficient of, say, $A_1$ follows the general
pattern~\eqref{asympt-alg}, but with $d=-3/2$~\cite[Thm. VII.7]{fla-sedg2}.
 The system is \emm strongly connected, if, roughly
speaking, the expression of every series $A_i$ involves (possibly
after a few iterations of the system) every other series $A_j$. For
instance, the system defining the walks ending at 0 in
Section~\ref{sec:line} reads 
$$
M_0= t^2(1+M_0)^2\quad \hbox{ and }\quad W_0=M_0(2+W_0).
$$
This system is not strongly connected, as $M_0$ does not involve
$W_0$. Accordingly, the number of $2n$-step walks returning to $0$
is ${{2n}\choose n} \sim \kappa 4^n n^{-1/2}$.

If one can rule out the possibility that $d=-5/2$ for $\ns$-algebraic
series, then this will prove that most map \gfs\ are not
$\ns$-algebraic (see the examples in Section~\ref{sec:world}).

\subsubsection{Some algebraic hypergeometric series}
Consider the following  series:
$$
F(t)=\sum_{n\ge 0 }f_nt^n= \sum_{n\ge 0 }\frac{\prod_{i=1}^d (a_in)!}{\prod_{j=1}^e (b_jn)!} t^n,
$$
where $a_1, \ldots , a_d, b_1, \ldots , b_e$ are positive
integers. This series is algebraic for some values of the $a_i$'s and
$b_j$'s, as shown by the case
$$
\sum_{n\ge 0 } \frac{(2n)!}{n!^2} t^n=\frac 1{\sqrt{1-4t}}.
$$
Can we describe all algebraic cases? Well, one can easily obtain some
necessary conditions on the sequences 
$a$ and $b$ by looking at the asymptotics of $f_n$. First, an
algebraic power series has a finite, positive radius 
of convergence (unless it is a polynomial). This, combined with
Stirling's formula, gives at once
\beq\label{necessary}
a_1+\cdots + a_d=b_1+\cdots +b_e.
\eeq
Moreover, by looking at the dominant term in the asymptotic behaviour
of $f_n$, and comparing with~\eqref{asympt-alg}, one obtains that
either $e=d$, or $e=d+1$. The case $d=e$ only gives the trivial solution $F(t)=1/(1-t)$, and the complete answer to this problem is as
follows~\cite{beukers,villegas}:
\begin{theorem}
  Assume~\eqref{necessary} holds, and $F(t) \not = 1/(1-t)$. The series $F(t) $ is algebraic if and only if
  $f_n \in \ns$ for all $n$ and $e=d+1$.
  \end{theorem}
Here are some algebraic instances:
$$
f_n= \frac{(6n)!(n)!}{(3n)!(2n)!^2}, \quad f_n= \frac{(10n)!(n)!}{(5n)!(4n)!(2n)!},
\quad f_n= \frac{(20n)!(n)!}{(10n)!(7n)!(4n)!}.
$$
The degree of these series is rather big: 12 [resp. 30]  for the first
[second] series above.
This theorem provides a collection of algebraic series with nice
integer coefficients: are these series $\ns$-algebraic? Do they count
some interesting objects?

\bigskip
\noindent {\bf Acknowledgements.} The parts of this survey that do not
deal exactly with the enumeration of combinatorial objects have
often be influenced by discussions with some of my colleagues, including
Fr\'ed\'erique Bassino, Henri Cohen, Philippe Flajolet, Fran\c cois Loeser,
G\'eraud S\'enizergues. Still, they should not be hold responsible for
any of the flaws of this paper...
\frenchspacing

\bibliographystyle{plain}
\bibliography{biblio.bib}

\end{document}